%% file: main.tex
\documentclass[12pt]{article}

\usepackage[english]{babel}
\usepackage[cp1251]{inputenc}
\usepackage[T2A]{fontenc}

\usepackage{hyperref}

\usepackage{amssymb}
\usepackage{amsmath}
\usepackage{amsthm} 
\usepackage{amsfonts}
\usepackage{graphicx,geometry}
\usepackage[title]{appendix}
\usepackage{cite}
\usepackage{float}
\usepackage{xcolor}
\usepackage[font=small,labelfont=bf]{caption}

\usepackage{tikz}
\usetikzlibrary{arrows,decorations.pathmorphing,backgrounds,positioning,fit,petri,patterns}

\renewcommand{\leq}{\leqslant}
\renewcommand{\geq}{\geqslant}

\DeclareFontEncoding{LS1}{}{}
\DeclareFontSubstitution{LS1}{stix2}{m}{n}

\newtheorem{theorem}{Theorem}
\newtheorem{lemma}{Lemma}
\newtheorem{proposition}{Proposition}
\newtheorem{corollary}{Corollary}

\theoremstyle{definition}
\newtheorem{definition}{Definition}
\newtheorem{remark}{Remark}

\newenvironment{enbibliography}{\vspace{-0.5cm}\begin{thebibliography}}{\end{thebibliography}}

\usepackage[title]{appendix}

\begin{document} 
\title{Oil displacement by slug injection: a rigorous justification for the Jouguet principle heuristic}
\author{Sergey Matveenko\footnote{ Chebyshev Laboratory, St. Petersburg State University, 14th Line V.O., 29, Saint Petersburg 199178 Russia. E-mail: matveis239@gmail.com.}, Nikita Rastegaev\footnote{St. Petersburg Department of Steklov Mathematical Institute
of Russian Academy of Sciences, 27 Fontanka, 191023, St. Petersburg, Russia. E-mail: rastmusician@gmail.com.}}
\renewcommand{\today}{}
\maketitle
\abstract{
In this paper we discuss a one-dimensional model for two-phase Enhanced Oil Recovery (EOR) floods, primarily for the polymer flood. We improve upon the method for the construction of semi-analytical solutions for the oil displacement by a water slug containing dissolved chemicals given in \cite{PiBeSh06} and later generalized in \cite{Pires2020, Pires2021}. This method utilizes a transformation into the Lagrange coordinates that splits the equations and allows one to solve the chromatographic one-phase problem separately. The solution is then substituted into a scalar hyperbolic conservation law, which is solved using the method of characteristics. However, there is often a gap in the characteristics near the chemical shock front. It was posited to the authors that the Jouguet principle could be used to close that gap. However, no rigorous justification was given for this approach, and as such it remained a heuristic. We analyze the conditions for the appearance of the gap and its properties, and give a proper argumentation for the Jouguet heuristic and its applicability based on the Kru\v{z}kov-type uniqueness theorem for the conservation law system. Additionally, a second splitting technique within the Lagrange coordinates is developed that simplifies this analysis and the construction of characteristics.

}

\section{Introduction}
The mathematical models for the Enhanced Oil Recovery (EOR) methods are often expressed as a system of hyperbolic conservation laws \cite{Gelfand, Dafermos}.
We study the conservation law system
\begin{equation}\label{eq:main_system_chem_flood}
\begin{cases} 
s_t + f(s, c)_x = 0, \\
(cs + a(c))_t + (cf(s,c))_x  = 0.
\end{cases}
\end{equation}
This one-dimensional two-phase three-component system is often used to describe the chemical flood of the oil reservoir. Here $(x,t)\in\mathbb{R}_+^2$, $s$ is the saturation of the water phase, $c$ is the concentration of the chemical agent dissolved in water, $f$ denotes the fractional flow function, usually S-shaped after Buckley--Leverett \cite{BL}, and $a$ describes the adsorption of the chemical agent on the rock, usually concave like the classical Langmuir curve (see Fig.~\ref{fig:BL_ads}).

We study the solutions of the slug injection initial-boundary value problem

\begin{align}
\label{eq:Initial_boundary_problem}
\begin{split}
s(x,0) &= s_{init}, \quad x > 0,\\
c(x,0) &= 0, \quad x > 0,\\
s(0,t) &= 1, \quad t \geqslant 0, \\
c(0,t) &= \begin{cases}
    c_{inj}, & 0 \leqslant t \leqslant t_{inj}, \\
    0, & t > t_{inj},
\end{cases}
\end{split}
\end{align}
and in particular the viability of the Jouguet principle in determining the unique solution. The techniques we describe and derive in this paper will work for all values of $s_{init}$ and $c_{inj}$, but for most proofs and examples we will assume $s_{init} = 0$, $c_{inj} = 1$. The case $s_{init} \neq 0$ has some differences in the area geometries and solution details, but nothing that changes the Jouguet principle analysis.

Note that the problem describing constant injection, for example
\begin{align*}
&s(x, 0) = c(x, 0) = 0, \quad x\geqslant 0, \\
&s(0, t) = c(0, t) = 1, \quad t\geqslant 0,
\end{align*}
is equivalent to the Riemann problem
\begin{equation*}
    (s,c)(x,0)=
    \begin{cases}
        (1,1),& \text{if } x\leq 0,\\
        (0,0),& \text{if } x>0.
    \end{cases}
\end{equation*}
The Riemann problem for the system \eqref{eq:main_system_chem_flood} was studied in \cite{JnW} and solutions for it are known for the class of fractional flow functions considered in this paper. The uniqueness of vanishing viscosity solutions for it was also considered in a certain context in \cite{Shen}. The slug injection problem was considered in \cite{PiBeSh06}. The Lagrange coordinate transformation was described in that paper to split the equations. This allows one to solve the chromatographic one-phase problem separately. The one-phase chromatographic problem is a fully solved problem for one chemical component, but even when considering multi-component floods, some partial cases (e.g.~problems with Langmuir adsorption) have known solutions (see \cite{RheeAmundson}). The solution is then substituted into a scalar hyperbolic conservation law describing phase movements, which is partially solved using the method of characteristics. If the solution of the chromatographic problem was continuous, this uplifting process would have been covered by the well-known Kru\v{z}kov's theorem \cite{Kruzhkov}, and the existence and uniqueness of the solution would be guarantied. However, the problems with shocks in the chemical concentrations are not covered by the same classical results. Moreover, there is usually a gap in the characteristics near the chemical shock front, that cannot be covered with just the characteristics constructed from the initial-boundary data. Additional assumptions must be made on some parts of such shocks for the method of characteristics to deliver a full solution. It was posited to the authors that the Jouguet principle could be used to close that gap, postulating that the characteristics incline must coincide with the shock speed, thus adding more initial data for characteristics in the gaps. However, no rigorous justification was given for this approach to date, and as such it remained a heuristic.

In the previous work \cite{MR2024} we used the Lagrange coordinate transformation described in \cite{PiBeSh06} to prove a uniqueness theorem similar to the classical Kru\v{z}kov's theorem \cite{Kruzhkov} for a general class of initial-boundary conditions, including problems involving, for example, more than one slug of a chemical agent or tapering \cite{Tapering}. With this theorem it is now feasible to justify the Jouguet principle when it is necessary for the solution construction, as well as study the conditions that create the gaps in the characteristics and predict the boundaries of the gaps a priori. Additionally, we developed a second splitting technique within the Lagrange coordinates that simplifies this analysis and the construction of characteristics, making it possible to only solve single equations instead of systems when constructing semi-analytical solutions algorithmically.

The paper has the following structure. Sect.~\ref{sec:restrictions} lists all restrictions we place on the parameters of the problem, i.e. on the flow function $f$ and on the adsorption function $a$. Sect.~\ref{sec:admissibility} defines the class of admissible solutions, describes the travelling wave dynamic system for the dissipative system, which provides the set of admissible shocks, and formulates the uniqueness theorem. Sect.~\ref{sec:Lagrange} briefly describes the Lagrange coordinate transformation. Sect.~\ref{sec:entropy} describes the mapping of shocks into the Lagrange coordinates and lists the admissibility criteria for the shocks. Sect.~\ref{sec:c-sol} describes the solution for the split equation on chemical concentration. Sect.~\ref{sec:solution-U} describes the solution for the other equation using the characteristics method. Here we introduce the new splitting technique for the characteristics equation system. We formulate the condition for the Jouguet principle applicability and use it to construct the solution and prove its uniqueness. Finally, Sect.~\ref{sec:Examples} contains figures demonstrating examples of solutions constructed using the presented method.

\section{Restrictions on problem parameters}
\label{sec:restrictions}

In this section we list the restrictions we put on the functions that are the parameters of our model, i.e. on functions $f$ (the fractional flow function) and $a$ (the adsorption curve). Before any explicit restrictions on fractional flow and adsorption, the model formulation \eqref{eq:main_system_chem_flood} inherently makes the following assumptions on the physics of the described process:
\begin{itemize}
    \item The flow takes place in a homogeneous porous media;
    \item The fluids are incompressible;
    \item Gravity effects are ignored;
    \item Second order effects, such as diffusion, dispersion and capillarity,  are small enough as to be negligible;
    \item The chemical component is dissolved only in the water phase;
    \item Water density does not change with chemical concentration.
\end{itemize}
These assumptions are reasonable for the description of an advection dominated horizontal flow.

\subsection{Restrictions on the flow function}
The following assumptions (F1)--(F4) for the fractional flow function $f$ are considered (see Fig.~\ref{fig:BL_ads}a for an example of function~$f$). 
\begin{enumerate}
    \item[(F1)] $f\in \mathcal C^2([0,1]^2)$; $f(0, c)=0$, $f(1, c)= 1$ for all $c\in[0,1]$;
    \item[(F2)] $f_s(s, c)>0$ for $0<s<1$, $0 \leq c \leq 1$;  $f_s(0,c)=f_s(1,c)=0$  for all $c\in[0,1]$;
    \item[(F3)] $f$ is $S$-shaped in $s$: for each $c \in [0,1]$ function $f(\cdot,c)$ has a unique point of inflection $s^I =  s^I(c) \in (0, 1)$, such that 
    \begin{itemize}
        \item[-] $f_{ss}(s, c)>0$ for $0<s<s^I(c)$,
        \item[-] $f_{ss}(s, c)<0$ for $s^I(c)<s<1$; 
    \end{itemize}
    
    \item[(F4)] $f$ is decreasing in $c$, i.e. $f_c(s,c)<0$ for all $s, c \in (0,1)$.

\end{enumerate}

\subsection{Restrictions on the adsorption function}
The adsorption function $a=a(c)$ satisfies the following assumptions (see Fig.~\ref{fig:BL_ads}b for an example of function~$a$):
\begin{itemize}
    \item[(A1)] $a \in \mathcal C^2([0,1])$, $a(0) = 0$;
    \item[(A2)] $a_c(c)>0$ for $0\leq c\leq1$;
    \item[(A3)] $a_{cc}(c)<0$ for $0\leq c\leq 1$.
\end{itemize}

\begin{figure}[htbp]
    \centering
    \includegraphics[width=0.4\textwidth]{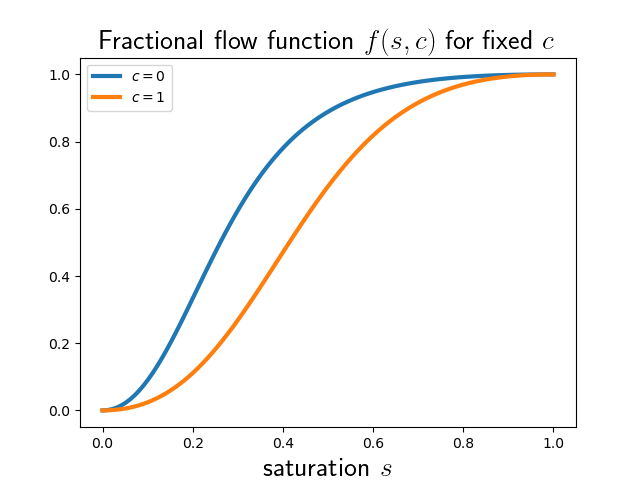}
    \includegraphics[width=0.4\textwidth]{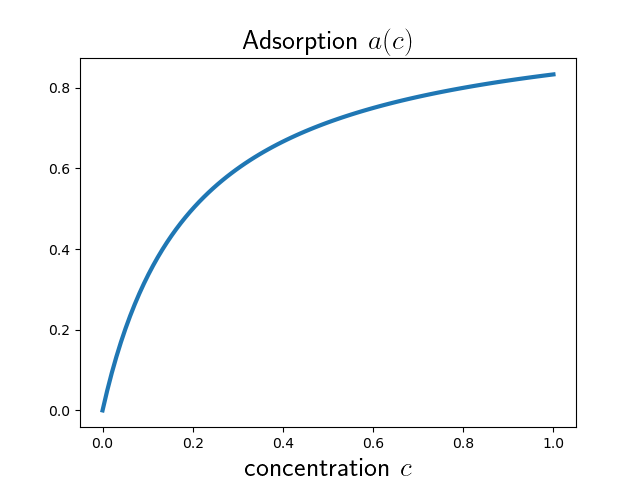}\\
    (a)\qquad\hfil\qquad  
    (b)\hfil
    \caption{Examples of (a) flow function $f(s,c)$;
    (b) adsorption function $a$.}
    \label{fig:BL_ads}
\end{figure}

\section{Admissible solutions of chemical flood system}
\label{sec:admissibility}

\subsection{Admissible weak solutions}
In order to use the uniqueness theorem proved in \cite{MR2024} we need to restrict the class of weak solutions with the following assumptions.

\begin{definition}\label{def:solution}
We call $(s, c)$ a piece-wise $\mathcal C^1$-smooth weak solution of \eqref{eq:main_system_chem_flood} with vanishing viscosity admissible shocks and locally bounded ``variation'' of $c$, if:
\begin{itemize}
    \item[(W1)] Functions $s$ and $c$ are continuous and piecewise continuously differentiable everywhere, except for a locally finite number of $\mathcal C^1$-smooth curves, where one or both of them have a jump discontinuity.
    \item[(W2)] For any compact $K$ away from the axes, the integral $\int_{t_1}^{t_2}|c_x(x,t)|\, dt < C_K$ is uniformly bounded for all $(x,t_1), (x,t_2)\in K$.
    \item[(W3)] Functions $s$ and $c$ satisfy \eqref{eq:main_system_chem_flood} in a classical sense inside the areas, where they are continuously differentiable.
    \item[(W4)] On every jump discontinuity curve $\Gamma$ given by $\gamma(t)$ at any point $(\gamma(t_0), t_0)$ the jump of $s$ and $c$ 
    \[
    s^\pm = s(\gamma(t_0)\pm 0, t), \quad c^\pm = c(\gamma(t_0)\pm 0, t)
    \]
    with velocity $v = \gamma_t(t_0)$ could be obtained as a limit as $\varepsilon\to 0$ of travelling wave solutions
    \[
    s(x,t) = \mathbf{s}\Big(\frac{x - vt}{\varepsilon}\Big), \quad c(x,t) = \mathbf{c}\Big(\frac{x - vt}{\varepsilon}\Big)  
    \]
    of the dissipative system
\begin{equation}\label{eq:main_system_dissipative}
\begin{cases} 
s_t + f(s, c)_x = \varepsilon s_{xx}, \\
(cs + a(c))_t + (cf(s,c))_x  = \varepsilon (c s_x)_x + \varepsilon c_{xx},
\end{cases}
\end{equation}
with boundary conditions
\[
\mathbf{s}(\pm\infty) = s^\pm, \quad \mathbf{c}(\pm\infty)  = c^\pm.
\]
\end{itemize}
\end{definition}

\begin{remark}
Note that the condition (W1) here is weaker than the condition (W1) in \cite{MR2024}. It allows $s$ and $c$ to have discontinuities in the derivative. Careful examination of the proofs in \cite{MR2024} shows that this changes nothing and the uniqueness theorem still holds in this wider class of solutions.
\end{remark}
\begin{remark}
Note also that the right-hand side in \eqref{eq:main_system_dissipative} represents the addition of small-to-negligible terms describing second-order effects, such as diffusion, dispersion and capillarity. However, as \cite[Section 3.1]{MR2024} asserts, under the restrictions (F1)--(F4) and (A1)--(A3) any reasonable change in the exact form of the right-hand side terms doesn't affect the set of admissible shocks the restriction (W4) describes. We could even add consideration for dynamic adsorption like in \cite[(3)]{Bahetal}, and the admissible shocks will stay the same. Therefore, we chose these terms for their simplicity and mathematical convenience, as they provide the clearest form of the travelling wave dynamic system \eqref{eq:dyn_sys_cap_diff} below.
\end{remark}

In this paper we only consider piece-wise $\mathcal C^1$-smooth weak solutions with vanishing viscosity admissible shocks and locally bounded ``variation'' of $c$, therefore, from now on we will just call them \emph{W-solutions} for brevity.

\subsection{Travelling wave dynamic system}
\label{sec:sec2-Hopf}
The assumption (W4) for the solution is that shocks are admissible if and only if they could be obtained as a limit of travelling wave solutions for a system with additional dissipative terms as these terms tend to zero. In this section we analyze such travelling wave solutions and derive a dynamic system that describes them.

Consider a shock between states $(s^-, c^-)$ and $(s^+, c^+)$ moving with velocity $v$. In order to check if it is admissible, we are looking for a travelling wave solution 
\[
s(x,t) = s\Big(\frac{x - vt}{\varepsilon}\Big), \quad c(x,t) = c\Big(\frac{x - vt}{\varepsilon}\Big)
\]
for the dissipative system \eqref{eq:main_system_dissipative}
satisfying the boundary conditions 
\[
s(\pm\infty) = s^\pm, \quad c(\pm\infty)  = c^\pm.
\]
Substituting this travelling wave ansatz into the system \eqref{eq:main_system_dissipative} and denoting $\xi = \frac{x - vt}{\varepsilon}$, we get the system
\begin{equation*}
\begin{cases} 
-v s_\xi + f(s, c)_\xi = s_{\xi\xi}, \\
-v (cs + a(c))_\xi + (cf(s,c))_\xi  = (c s_\xi)_\xi + c_{\xi\xi}.
\end{cases}
\end{equation*}
Integrating the equations over $\xi$ we arrive at the travelling wave dynamic system
\begin{equation}\label{eq:dyn_sys_cap_diff}
\begin{cases} 
s_\xi = f(s, c) - v (s + d_1), \\
c_\xi = v (d_1 c - d_2 - a(c)).
\end{cases}
\end{equation}
The values of $d_1$ and $d_2$ are obtained from the boundary conditions:
\begin{align*}
    vd_1 & = -vs^\pm + f(s^\pm, c^\pm), \\
    vd_2 & = v d_1 c^\pm - v a(c^\pm),
\end{align*}
namely, in the case when $c^+ \neq c^-$,
\begin{equation*}
d_1 = \dfrac{a(c^-) - a(c^+)}{c^- - c^+}, \quad d_2 =  \dfrac{c^+ a(c^-) - c^- a(c^+)}{c^- - c^+}.
\end{equation*}
Additionally, the same boundary conditions yield us the Rankine--Hugoniot conditions
\begin{equation}
\label{eq:RH-1}
\begin{split}
    v[s]&=[f(s,c)],
    \\
    v[cs+a(c)]&=[cf(s,c)],
\end{split}
\end{equation}
where $[q(s,c)]=q(s^+,c^+)-q(s^-,c^-)$\footnote{Note the order of ``$+$'' and ``$-$'' terms in this definition. It could be different in different sources. We follow certain proof schemes of \cite{Serre1}, so our order coincides with their.}. Thus, for every set of shock parameters $(s^{\pm}, c^{\pm})$ and $v$ satisfying \eqref{eq:RH-1}, we can construct a phase portrait for the dynamic system \eqref{eq:dyn_sys_cap_diff}. The points $(s^\pm, c^\pm)$ are critical for this dynamic system due to \eqref{eq:RH-1}, and we can check if there is a trajectory connecting the corresponding critical points. But even just analyzing the geometric meaning of the Rankine--Hugoniot conditions \eqref{eq:RH-1}, we derive a lot of restrictions on admissible shock parameters.

\begin{proposition}[Proposition 3.2 \cite{MR2024}]\label{prop:inadmissible_shocks}
The following restrictions on admissibility are evident from the properties (F1)--(F4), (A1)--(A3), the Rankine--Hugoniot conditions \eqref{eq:RH-1} and the analysis of the sign of the right-hand side of \eqref{eq:dyn_sys_cap_diff}:
\begin{itemize}
    \item Admissible shock velocity $v$ is bounded and strictly positive: $0 < v < \|f\|_{\mathcal C^1}$.
    \item Shocks with $s^- = 0$ cannot be admissible.
    \item Shocks with $s^+ = s^-$ cannot be admissible.
    \item Shocks with $c^+ > c^-$ cannot be admissible.
    \item If $s^+ = 0$ then $c^+=c^-$.
\end{itemize}
\end{proposition}

\subsection{Admissible $s$-shocks}
\label{sec:sec2-s-shocks}
Consider an $s$-shock, i.e. a shock with no change in $c = c^+ = c^-$. The system \eqref{eq:dyn_sys_cap_diff} for this case simplifies into one equation
\[
s_\xi = f(s, c) - f(s^-, c) - v (s - s^-) =: \Psi(s).
\]
For this equation $s^\pm$ are critical points due to the Rankine--Hugoniot condition \eqref{eq:RH-1}. And the existence of a trajectory connecting these critical points depends on the sign of the right-hand side of the equation between the critical points. For the trajectory to exist it is necessary and sufficient that
\begin{equation}\label{Oleinik_admissibility}
\Psi(s)(s^+-s^-) \geqslant 0 \quad \text{ for all } s \text{ between } s^+ \text{ and } s^-.
\end{equation}
This condition is called Oleinik's entropy condition or E-condition \cite{Oleinik}. Note that since $f$ is $S$-shaped, it is impossible for $\Psi(s)$ to be $0$ inside the interval without changing sign, therefore, the sign in this condition could be interpreted as strict without loss of generality. Note also that $\Psi_s(s) = f_s(s, c) - v$. Therefore, for the current problem this criterion is equivalent to the following pair of conditions:
\begin{itemize}
    \item $\Psi(s) \neq 0$ for all $s$ between $s^+$ and $s^-$;
    \item $f_s(s^+, c) \leqslant v \leqslant f_s(s^-, c)$, but both signs cannot be equal at the same time.
\end{itemize}
The inequalities in the second condition of the pair are known as Lax's shock condition.

\subsection{Admissible $c$-shocks}
\label{sec:sec2-nullclines}

In this section we provide the results of \cite[Sect.~3.3]{MR2024}.

Based on nullcline configuration classification it is easy to trace the possible and impossible trajectories for the travelling wave dynamical system \eqref{eq:dyn_sys_cap_diff} and name admissible and inadmissible $c$-shocks (shocks with $c^+\neq c^-$). Since any possible travelling wave must satisfy the Rankine--Hugoniot condition \eqref{eq:RH-1} and we have the S-shaped condition (F3), for any given $v$, $c^\pm$ there could be at most 2 possible values of $s^+$ and 2 values of $s^-$ that satisfy it. When there are 2 different values for $s^-$, there will be 2 different values for $s^+$ due to monotonicity (F4). In this case we number them in increasing order: $u_{1, 2}^\pm = (s_{1,2}^\pm, c^\pm)$, $s_1^\pm < s_2^\pm$. Due to the Rankine--Hugoniot condition \eqref{eq:RH-1} it is clear that $u_{1,2}^\pm$ represent the critical points of the dynamical system \eqref{eq:dyn_sys_cap_diff}.
\begin{proposition}
\label{prop:c-shock-admissibility}
The local vanishing viscosity condition (W4) imposes the following restrictions on the $c$-shocks:
\begin{itemize}
    \item As already noted in Proposition \ref{prop:inadmissible_shocks}, no trajectories could exist from $c^-$ to $c^+$ if $c^- < c^+$, therefore such shocks are inadmissible.
    \item When $u_1^- \neq u_2^-$ exist, travelling waves connecting $u_2^-$ to $u_1^+$ are inadmissible, since there is no trajectory connecting those critical points.
    \item All other travelling waves between critical points $u^-_{1,2}$ and $u^+_{1,2}$ are admissible.
\end{itemize}
\end{proposition}

\subsection{The uniqueness theorem}
\label{sec:uniqueness}

In his famous 1970 paper \cite{Kruzhkov} Kru\v{z}kov proved the weak solution uniqueness theorem for scalar conservation laws. We followed the scheme of his proof as explained in \cite{Serre1} to prove a similar uniqueness theorem for the polymer injection system \eqref{eq:main_system_chem_flood}.

\begin{theorem}[Theorem 6.1, \cite{MR2024}]
\label{uniqueness-theorem}
Problem \eqref{eq:main_system_chem_flood} with initial-boundary conditions \eqref{eq:Initial_boundary_problem} satisfying the restrictions (S1)--(S3), with flow function satisfying (F1)--(F4) and adsorption satisfying (A1)--(A3) can only have a unique W-solution.
\end{theorem}

\section{Lagrange coordinate transformation}
\label{sec:Lagrange}

In this section we recall the results of \cite[Sect.~4]{MR2024} and adapt them to the case of boundary conditions \eqref{eq:Initial_boundary_problem}. For brevity we omit all proofs and most auxiliary lemmas and propositions.

\subsection{Lagrange coordinates}
\label{sec2-Lagrange}
The history of the Lagrange coordinate transformation in hyperbolic conservation laws dates back at least to the first half of the XX century, to the work \cite{Courant} cited in \cite{Wa87} in the context of gas dynamics equations. The idea is also presented in the lectures by Gelfand \cite{Gelfand} for the case of an arbitrary system of conservation laws. The splitting technique using the Lagrange coordinate transformation we apply to the system \eqref{eq:main_system_chem_flood} is presented in \cite{PiBeSh06}. It is later developed and applied to different systems by many authors (see \cite{Pires2021} and references therein). Previously we used this transformation to prove the uniqueness theorem in \cite{MR2024} and provide the formal proof of the solution mapping. We also studied the area where $s=0$, which is where the transformation breaks down and becomes non-smooth.

\begin{lemma}[Lemma 4.4, \cite{MR2024}]
\label{lemma:t0_is_shocks}
For all $x > 0$ we define
\[
t_0(x) = \sup\{ t: s(x, t) = 0 \}.
\]
Then 
\begin{itemize}
    \item $t_0(x) < +\infty$;
    \item $(x, t_0(x))$ is a point on a shock;
    \item $t_0(x)$ is continuous, piece-wise $\mathcal{C}^1$-smooth.
\end{itemize}
\end{lemma}

\begin{corollary}[Corollary 4.5, \cite{MR2024}]\label{corollary:zero_flow_area}
Define $\Omega_0 = \{(x,t): x > 0, 0\leqslant t < t_0(x)\}$. Then $s(x,t) = 0$ in $\Omega_0$ and $s(x,t) > 0$ outside $\overline{\Omega}_0$. Moreover, $s(x,t)$ is locally separated from $0$ outside $\overline{\Omega}_0$.
\end{corollary}

\begin{proposition}[Proposition 4.6, \cite{MR2024}]\label{prop:c_on_zero_boundary}
$c_t = 0$ in $\Omega_0$, therefore $c(x, t_0(x)) = 0$.
\end{proposition}

In the case of boundary conditions \eqref{eq:Initial_boundary_problem} we note that when $s_{init} > 0$ we have $t_0(x) \equiv 0$, and $\Omega_0$ is empty. On the other hand, when $s_{init} = 0$, we have $t_0(x) > 0$ for all $x>0$.

We denote by $\varphi$ the potential such that
\begin{equation} 
\label{eq:dPhi}
    d\varphi=f(s,c)\,dt-s\,dx.
\end{equation}
To explain the physical meaning of $\varphi$ let us consider any trajectory $\nu$ connecting $(0, 0)$ and $(x, t)$.
When $s$ denotes the saturation of some liquid, the potential $\varphi(x, t)$ is equal to the amount of this liquid
passing through the trajectory:
\begin{equation}
\label{eq:phi}
    \varphi(x,t)=\int\limits_{\nu} f(s,c)\,dt-s\,dx.
\end{equation}

This coordinate change is only applicable in the area $Q_{orig} = Q\setminus \overline{\Omega}_0$, where the saturation $s$ and the flow function $f(s,c)$ are not zero. It keeps the $x$ coordinate, so it maps the axis $\Gamma_t=\{(0,t): t>0\}$ onto itself. 
When $s_{init} > 0$ (see Fig.~\ref{fig:orig-lagr-areas-not-0}), the axis $\Gamma_x=\{(x,0): x>0\}$, which coincides with the curve $(x, t_0(x) = 0)$ for $x > 0$, maps into a ray $( \varphi_0(x), x)$, where
\begin{equation}
\label{eq:def_varphi_0}
\varphi_0(x) = -\int\limits_0^x s(r, 0) \, dr = - x s_{init}.
\end{equation}
When $s_{init} = 0$ (see Fig.~\ref{fig:orig-lagr-areas-0}), the curve $(x, t_0(x))$ for $x > 0$ maps into $\Gamma_x$, which can also be represented as a ray with  $\varphi_0(x) = - x s_{init} = 0$. Therefore, in both cases  $Q_{orig}$ maps into 
\[
Q_{lagr} = Q \cup \{ (x, \varphi) : 0 < x < +\infty, -x s_{init} < \varphi < +\infty  \}.
\]

Corollary \ref{corollary:zero_flow_area} guaranties that there is a inverse transform given by
\[
dt=\frac1{f(s,c)}\,d\varphi+\frac s{f(s,c)}\,dx,
\]
and the denominators are locally separated from zero, therefore this coordinate change is a piecewise $\mathcal C^1$-diffeomorphism.
Since $\mathcal C^1$-smooth curves preserve their smoothness properties under any diffeomorphism, all discontinuity curves map into $\mathcal C^1$-smooth discontinuity curves.

When applied to the system \eqref{eq:main_system_chem_flood} in weak form (for mathematically rigorous proof of transformation see \cite[Sect.~4.2]{MR2024}) this coordinate change gives us inside the areas of $\mathcal C^1$-smoothness the classical system
\begin{figure}[ht!]
    \begin{minipage}{\linewidth}
    \begin{center}
    \include{pic_Q_lagr_Q_orig_2.tex}
    \end{center}
    \end{minipage}
    \caption{Areas $Q_{orig}$ (grey area on the left), and $Q_{lagr}$ (grey area on the right) in the case $s_{init} \neq 0$. The blue line on the left is mapped onto the blue line $\varphi_0(x) = -x s_{init}$ on the right.}
    \label{fig:orig-lagr-areas-not-0}
    \begin{minipage}{\linewidth}
    \begin{center}
    \include{pic_Q_lagr_Q_orig.tex}
    \end{center}
    \end{minipage}
    \caption{Areas $Q_{orig}$ (grey area on the left), and $Q_{lagr}$ (grey area on the right) in the case $s_{init} = 0$. The red curve on the left is mapped onto the red line on the right.}
    \label{fig:orig-lagr-areas-0} 
\end{figure}
\begin{align*}
\begin{split}
    \frac{\partial}{\partial x}\left(\frac1f\right) - \frac{\partial}{\partial \varphi}\left(\frac s f\right)&=0, \\
    \frac{\partial c}{\partial x} + \frac{\partial a(c)}{\partial \varphi}&=0. 
\end{split}
\end{align*}
We use the notation
\begin{equation} 
\label{eq:def-U-F}
\mathcal{U}(\varphi, x)=\frac1{f(s(x,t),c(x,t))}, \quad \zeta(\varphi,x) = c(x,t), \quad\text{and}\quad\mathcal F(\mathcal U,\zeta)=-\frac{s}{f(s,c)}
\end{equation}
in order to transform this system into the system of conversation laws
\begin{align}
       \mathcal U_x + \mathcal F(\mathcal U,\zeta)_\varphi & = 0,\label{eq:U-lagr-eqn}\\
    \zeta_x + a(\zeta)_\varphi&=0. \label{eq:c-lagr-eqn}
\end{align}
\begin{remark}
The same reasoning as in \cite[Lemma 2.2.1]{Serre1} shows that on every shock the equations in the weak form result in the Rankine--Hugoniot condition
\begin{equation} 
\label{eq:RH-2}
\begin{split}
    v^*[\mathcal U]&=[\mathcal F(\mathcal U,\zeta)],
    \\
    v^*[\zeta]&=[a(\zeta)],
\end{split}
\end{equation} 
where $v^*$ is the velocity of the shock between states $(\mathcal U^-, \zeta^-)$ and $(\mathcal U^+, \zeta^+)$. Here, like in original coordinates, $[q(\mathcal U, \zeta)]=q(\mathcal U^+, \zeta^+)-q(\mathcal U^-, \zeta^-)$.
\end{remark}

Properties of the new flow function $\mathcal F$ (see Fig.~\ref{fig:BLf_lagr}) that correspond to the properties (F1)--(F4) of the function $f$ are listed below.
\begin{figure}[htbp]
    \centering
    \includegraphics[width=0.55\textwidth]{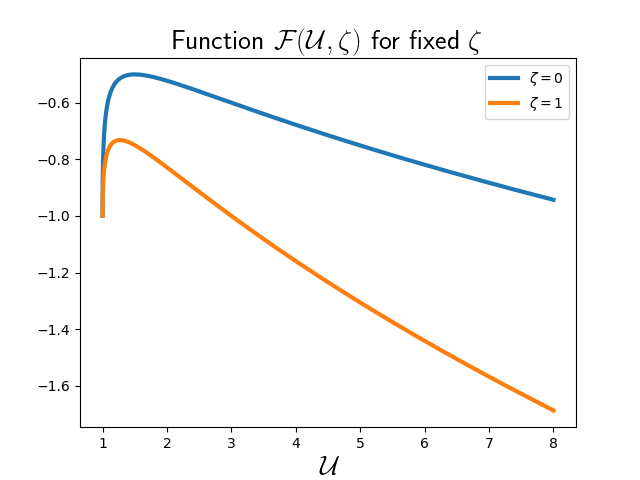}
    \caption{The function $\mathcal F(\mathcal U,\zeta)$ corresponding to the flow function $f(s, c)$ plotted in Fig \ref{fig:BL_ads} (a).}
    \label{fig:BLf_lagr}
\end{figure}
\begin{proposition}[Proposition 4.9, \cite{MR2024}]
\label{prop:new-flow-function-properties}
For all $\zeta\in[0,1]$ the following properties of the function $\mathcal F$ are fulfilled
\begin{itemize}
\item[($\mathcal F$1)]
\begin{itemize}
\item[$\bullet$]
$\mathcal F \in \mathcal C^2(( 1,+\infty)\times[0,1])$; 
\item[$\bullet$]
$\mathcal F(\mathcal U,\zeta) < 0$ for all $\,\mathcal U\in [1,+\infty)$;
\item[$\bullet$]
$\mathcal F(1, \zeta) = - 1$;
\item[$\bullet$]
$\lim\limits_{\mathcal U\to\infty}\mathcal{F}(\mathcal U, \zeta) = -\infty$;
\end{itemize}
 \medskip
\item[($\mathcal F$2)]
\begin{itemize}
\item[$\bullet$]
$\lim\limits_{\mathcal U\to1}\mathcal{F}_{\mathcal U}(\mathcal U, \zeta) = +\infty$;
\item[$\bullet$]
$\lim\limits_{\mathcal U\to\infty}\mathcal{F}_\mathcal{U}(\mathcal U, \zeta) = 0$;
\item[$\bullet$]
there exists a unique global maximum $\mathcal U^{\max}(\zeta)$;
\item[$\bullet$]
 $\mathcal F_{\mathcal U}(\mathcal U, \zeta) > 0$ for $\mathcal U \in (1, \mathcal U^{\max}(\zeta))$; 
 \item[$\bullet$]
 $\mathcal F_{\mathcal U}(\mathcal U, \zeta) < 0$ for
 $\mathcal U \in (\mathcal U^{\max}(\zeta), +\infty)$;
 \end{itemize}
 \medskip
 \item[($\mathcal F$3)]
 \begin{itemize}
\item[$\bullet$] 
function $\mathcal F(\cdot,\zeta)$ has a unique point of inflection $\mathcal U^I(\zeta) \in (1, +\infty)$, such that $\mathcal F_{\mathcal U\mathcal U} (\mathcal U, \zeta) < 0$ for $1 < \mathcal U < \mathcal U^I(\zeta)$ and $\mathcal F_{\mathcal U\mathcal U} (\mathcal U, \zeta) > 0$ for $\mathcal U > \mathcal U^I(\zeta)$. \item[$\bullet$] $\mathcal U^{I}(\zeta) > \mathcal U^{\max}(\zeta)$;
\item[$\bullet$]
therefore, if $\mathcal U$, $\zeta$ are such that $\mathcal F_{\mathcal U} (\mathcal U, \zeta) > 0$, then $\mathcal F_{\mathcal U\mathcal U} (\mathcal U, \zeta) < 0$;
\end{itemize}
\item[($\mathcal F$4)]
function 
$\mathcal F$ is monotone in $\zeta$, i.e.  for all $\mathcal U\in(1, +\infty)$ we have $\mathcal F_\zeta(\mathcal U, \zeta) < 0$. \end{itemize}
\end{proposition}

\section{Entropy conditions in Lagrange coordinates}
\label{sec:entropy}

\subsection{Mapping shocks to Lagrange coordinates}
\label{sec:sec5-shock-mapping}
When constructing the solution in Lagrange coordinates, we need to discern which shocks are admissible. One possible approach (see e.g. \cite{Shen}) is to use the vanishing viscosity method directly for the system \eqref{eq:U-lagr-eqn}--\eqref{eq:c-lagr-eqn} to establish admissibility criteria for the shocks. However, we consider that approach flawed. It adds second-order terms into the equation system in Lagrange coordinates, but those terms lack any kind of physical significance. In some simple cases it could be argued that the set of admissible shocks is the same regardless of how we add dissipative terms, but it was shown in \cite{Bahetal} that even for the original system with a more complex dependence of the flux function on $c$ we have different admissible shocks for different ratios of dissipative parameters. Therefore, the best approach is to choose the most physically meaningful form of dissipative terms to tie the admissibility of shocks to something that could be experimentally measured in real systems.

That is why we established admissibility in original coordinates with dissipative system \eqref{eq:main_system_dissipative}. Now, in order to define admissibility in Lagrange coordinates, we use \eqref{eq:def-U-F} to construct a map between shocks in original coordinates and shocks in Lagrange coordinates. Only shocks corresponding to vanishing viscosity admissible shocks in original coordinates will be considered admissible in Lagrange coordinates. This mapping goes as follows. 

Consider a shock in original coordinates at the point $(x_1, t_1)$ with values
\[
s^\pm = s(x_1\pm 0, t_1), \quad c^\pm = c(x_1\pm 0, t_1).
\]
Denote $\theta_c(s) = \frac{1}{f(s,c)}$ and $\vartheta_c = \theta_c^{-1}$ its inverse function with respect to its argument $s$. Using these functions we map
\[
\mathcal U^{[+]} = \theta_{c^+}(s^+), \quad \mathcal U^{[-]} = \theta_{c^-}(s^-), \quad \zeta^{[+]} = c^+, \quad \zeta^{[-]} = c^-,
\]
\[s^+ = \vartheta_{\zeta^{[+]}}(\mathcal U^{[+]}), \quad s^- = \vartheta_{\zeta^{[-]}}(\mathcal U^{[-]}), \quad c^+ = \zeta^{[+]}, \quad c^- = \zeta^{[-]}. 
\]
Note, that in the original coordinates values $s^\pm$ correspond to $x\to x_1\pm 0$ respectively. The shock velocity in original coordinates is always positive due to Proposition \ref{prop:inadmissible_shocks}, therefore, $s^\pm$ correspond to $t\to t_1\mp 0$:
\[
s^\pm = s(x_1, t_1 \mp 0), \quad c^\pm = c(x_1, t_1\mp 0).
\]
Further, due to \eqref{eq:dPhi} and (F1), when $x = x_1$ is fixed, $t\to t_1\mp 0$ correspond to $\varphi\to \varphi_1\mp 0$ for the point $(\varphi_1, x_1)$ on a corresponding shock in Lagrange coordinates, and so do $\mathcal U^{[\pm]}$:
\[
\mathcal U^{[\pm]} = \mathcal U(\varphi_1 \mp 0, x_1), \quad \zeta^{[\pm]} = \zeta(\varphi_1 \mp 0, x_1).
\]
But for the equations \eqref{eq:U-lagr-eqn}, \eqref{eq:c-lagr-eqn} in Lagrange coordinates the $x$ axis plays the role of time and $\varphi$ the role of space, so we would like to denote $\mathcal U^\pm$ to correspond to $\varphi\to \varphi_1\pm 0$. Thus, we denote
\[
\mathcal U^+ = \mathcal U^{[-]}, \quad \mathcal U^- = \mathcal U^{[+]}, \quad \zeta^+ = \zeta^{[-]}, \quad \zeta^- = \zeta^{[+]},
\]
\[
\mathcal U^{\pm} = \mathcal U(\varphi_1 \pm 0, x_1), \quad \zeta^{\pm} = \zeta(\varphi_1 \pm 0, x_1),
\]
and obtain a one-to-one mapping of shocks in original and Lagrange coordinates:
\begin{equation}\label{shock_mapping}
(\mathcal U^\pm, \zeta^\pm) \to (\vartheta_{\zeta^\mp}(\mathcal U^\mp), \zeta^\mp), \quad (s^\pm, c^\pm) \to (\theta_{c^\mp}(s^\mp), c^\mp).
\end{equation}

\subsection{Oleinik, Lax and entropy admissibility for $\mathcal U$-shocks}
\label{sec:sec2-s-shock-admissibility}

We transfer the Oleinik and Lax conditions described in Sect.~\ref{sec:sec2-s-shocks} to the Lagrange coordinates, obtaining similar inequalities for $\mathcal F(\mathcal U, \zeta)$, $\zeta = c$ and
\[
\Psi^*(\mathcal U) = \mathcal F(\mathcal U, \zeta) - \mathcal F(\mathcal U^-, \zeta) - v^*(\mathcal U - \mathcal U^-),
\]
\[
v^* = 
\begin{cases}
\dfrac{\mathcal F(\mathcal U^-, \zeta) - \mathcal F(\mathcal U^+, \zeta)}{\mathcal U^- - \mathcal U^+}, & \mathcal U^- < +\infty,\\
0, & \mathcal U^- = +\infty.
\end{cases}
\]
Similar to the original coordinates, $\mathcal U^\pm$ are zeroes of $\Psi^*$ due to the exact formula for $v^*$ obtained from Rankine--Hugoniot conditions \eqref{eq:RH-2}.

\begin{lemma}[Lemma 5.1, \cite{MR2024}]
\label{lemma_Oleinik_equivalence}
Lax condition for an $s$-shock in original coordinates is equivalent to Lax condition for the corresponding $\mathcal U$-shock in Lagrange coordinates:
\begin{equation}\label{eq:Lax-Lagrange}
\mathcal F_{\mathcal U}(\mathcal U^+, \zeta) \leqslant v^* \leqslant \mathcal F_{\mathcal U}(\mathcal U^-, \zeta),
\end{equation}
and both signs can only be equal at the same time when $\mathcal U^- = +\infty$.

Moreover, $\Psi^*(\mathcal U) \neq 0$ if and only if for the corresponding $s$ we have $\Psi(s) \neq 0$, therefore the Oleinik E-condition is also equivalent for $s$-shocks in original coordinates and $\mathcal U$-shocks in Lagrange coordinates.
\end{lemma}

\subsection{Admissibility of $\zeta$-shocks}

Recall that due to Proposition \ref{prop:inadmissible_shocks} (and also Proposition \ref{prop:c-shock-admissibility}) for an admissible shock with different values of $c$ we must have $c^- > c^+$ and thus $\zeta^- < \zeta^+$. Therefore, due to the concavity of the function $a$ we always have Lax condition for the equation \eqref{eq:c-lagr-eqn}.
\begin{proposition}[Proposition 5.7, \cite{MR2024}]
For any admissible $\zeta$-shock we have Lax condition
\[
a_\zeta(\zeta^+) < \dfrac{a(\zeta^-)-a(\zeta^+)}{\zeta^--\zeta^+} < a_\zeta(\zeta^-).
\]
Moreover, 
\[
\dfrac{a(\zeta)-a(\zeta^-)}{\zeta-\zeta^-} > \dfrac{a(\zeta^-)-a(\zeta^+)}{\zeta^--\zeta^+} \quad \text{ for all } \zeta\in(\zeta^-, \zeta^+),
\]
therefore, Oleinik's E-condition also holds.
\end{proposition}

To verify the admissibility of $\zeta$-shocks in the first equation, we utilize the mapping provided in Sect.~\ref{sec:sec5-shock-mapping} and check the admissibility with Proposition~\ref{prop:c-shock-admissibility}.

\section{Slug injection solution for $\zeta$}
\label{sec:c-sol}

After passing to the Lagrangian coordinates, in view of Proposition \ref{prop:c_on_zero_boundary}
the initial and boundary conditions for $\zeta$ are the same as they were for $c$:
\begin{equation}
    \label{eq:c-init-and-bound-cond}
    \zeta(0, x)=0,\qquad \zeta(\varphi,0)=\begin{cases}
    1,\qquad \varphi\in(0,t_{inj}],\\
    0,\qquad \varphi\in(t_{inj},+\infty).
    \end{cases}
\end{equation}

Since the equation \eqref{eq:c-lagr-eqn} does not depend on $\mathcal U$, we can solve it separately from~\eqref{eq:U-lagr-eqn}. 
The $\zeta$-characteristics are straight lines satisfying the equation
\begin{equation*} 
    \varphi(x) = \varphi_0 + a_\zeta(\zeta(\varphi_0, x_0)) (x - x_0).
\end{equation*} 
Along each of these lines the value of $\zeta$ is a constant: 
$$\zeta(\varphi(x), x) = \zeta(\varphi_0, x_0).$$
The function $a_\zeta$ is positive and strictly decreasing on $[0, 1]$ due to (A2)--(A3), so it is easy to describe the solution of \eqref{eq:c-lagr-eqn} in detail.
\begin{figure}[H]
    \begin{center}
    \includegraphics[width=0.8\linewidth]{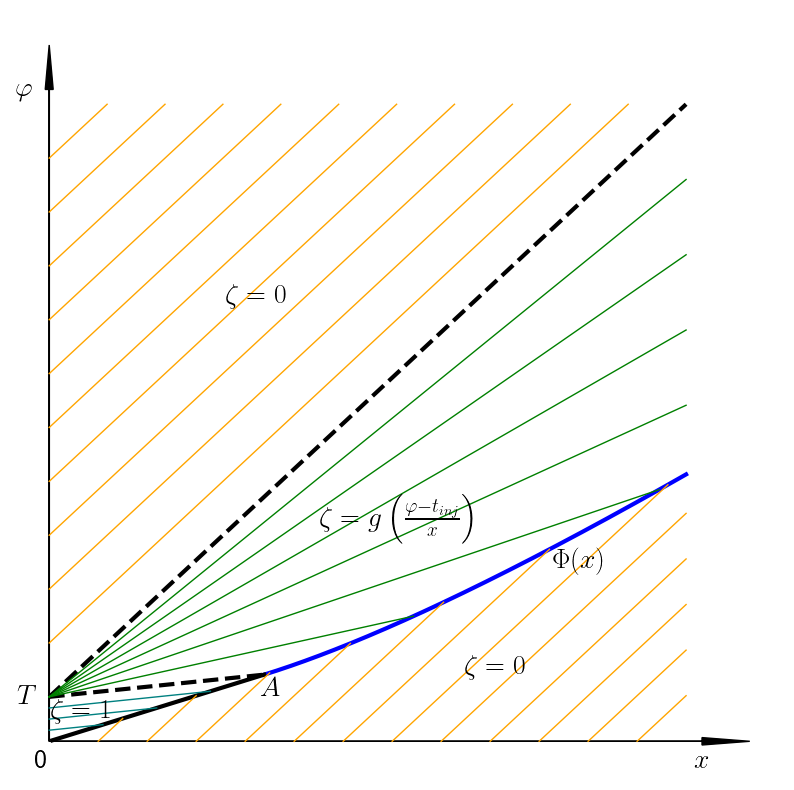}
    \end{center}
    \caption{Characteristics of the $\zeta$-solution.}
    \label{fig:c_solution_scheme}
\end{figure}
We use the notation
$$v(\zeta_1, \zeta_2) = \frac{a(\zeta_1) - a(\zeta_2)}{\zeta_1 - \zeta_2}.$$
While $\varphi < t_{inj}$, the solution of the slug injection problem coincides with the solution of the Riemann problem. The initial-boundary conditions \eqref{eq:c-init-and-bound-cond} guarantee that the origin is a discontinuity point with $(\zeta^+, \zeta^-) = (1, 0)$. Hence, there is a $\zeta$-shock starting from the origin and moving with the velocity $v(1, 0)$ due to the Rankine--Hugoniot condition \eqref{eq:RH-2} hereinafter referred to as the \emph{chemical shock front}.
The slope of characteristics below the shock is~$a_\zeta(0)$ and the slope above the shock is $a_\zeta(1)$ up to a certain point.
The upper characteristic with the slope equal to $a_\zeta(1)$ starts from the point~$T = (t_{inj}, 0)$ and satisfies the equation
$\varphi(x) = t_{inj} + a_\zeta(1)x$.
The characteristics above it corresponds to a rarefaction wave and satisfy $\varphi = t_{inj} + a_\zeta(\zeta)x$, $\zeta\in[0,1)$.
So if we introduce an inverse function
\begin{equation}\label{eq:g_definition}
g = (a_\zeta)^{-1},
\end{equation}
we can express $\zeta$ within the conical domain of the rarefaction wave before the characteristics reach the chemical shock front as follows:
\begin{equation}\label{eq:zeta_in_cone_definition}
\zeta(\varphi, x) = g\left(\frac{\varphi-t_{inj}}x\right).
\end{equation}

It is clear that a part of the chemical shock front shock near the origin is a segment with the slope $v(1, 0)$. Denote by $A = (\varphi_A, x_A)$ the right end of this segment.
The point $A$ is an endpoint of the $\zeta$-characteristic starting from $T$ carrying the value $\zeta=1$, and thus having the slope $a_\zeta(1)$ (the blue segment on Fig. \ref{fig:c_solution_scheme}). 
We obtain the coordinates of $A$ as the intersection point of two lines ($\varphi = v(1, 0)x$ and $\varphi = t_{inj} + a_\zeta(1)x$):
\begin{equation}
\label{eq:x_A_phi_A}
x_A = \frac{t_{inj}}{v(1, 0) - a_\zeta(1)},\qquad 
\varphi_A = \frac{v(1, 0)t_{inj}}{v(1, 0)-a_\zeta(1)}.
\end{equation}
Denote by $\Phi(x)$ the curved part of the chemical front shock to the right of $x_A$. This part of the shock corresponds to the trajectory of slug front after the maximum concentration becomes less than $1$. Due to \eqref{eq:zeta_in_cone_definition} we have $\zeta(\Phi(x)+0, x) = g\left(\frac{\Phi(x)-t_{inj}}{x}\right)$ above the shock and $\zeta(\Phi(x)-0, x) = 0$ below the shock. Therefore, the Rankine--Hugoniot condition implies that the curve $\Phi$ satisfies the differential equation 
\begin{equation}
\label{eq:Phi_equation}
\frac{d\Phi(x)}{dx} = v\left(g\left(\frac{\Phi(x)-t_{inj}}x\right), 0\right),\quad x>x_A
\end{equation}
with the initial condition $\Phi(x_A) = \varphi_A$. This differential equation was previously solved analytically in some specific cases, e.g. for $a$ given by the Langmuir curve. But it can be integrated even for a general concave $a$. 

\begin{lemma}
\label{lemma_explicit_Phi}
Denote $p(\zeta) = a(\zeta)-\zeta a_\zeta(\zeta)$ and $q = p^{-1}$ its inverse function. Then
\[
\Phi(x) = t_{inj} + a_\zeta\Big(q\Big(\dfrac{t_{inj}}{x}\Big)\Big) x.
\]
\end{lemma}
\begin{proof}
Denote $\zeta_\Phi(x) = g\left(\frac{\Phi(x)-t_{inj}}x\right)$. Using the definition \eqref{eq:g_definition} of the function $g$ and \eqref{eq:Phi_equation}, we derive
\[
\dfrac{d}{dx}\zeta_\Phi(x) = \dfrac{p(\zeta_\Phi(x))}{x\zeta_\Phi(x)a_{\zeta\zeta}(\zeta_\Phi(x))}.
\]
This equation splits and integrates trivially.
\[
\int\limits_{\zeta_\Phi(x)}^1 \dfrac{-ra_{\zeta\zeta}(r)}{p(r)}\,dr = \ln\dfrac{x}{x_A}.
\]
Note that $p_\zeta(\zeta) = -\zeta a_{\zeta\zeta}(\zeta)$, so the integral in the left-hand side can also be calculated explicitly after changing the variable (note also that the derivative is positive, therefore the inverse function $q$ exists). Thus, we arrive at
\[
\dfrac{p(1)}{p(\zeta_\Phi(x))} = \dfrac{x}{x_A}.
\]
Note also that by the definition \eqref{eq:x_A_phi_A} of $x_A$ we have $x_A p(1) = t_{inj}$, therefore we obtain
\[
\zeta_\Phi(x) = q\left(\dfrac{t_{inj}}{x}\right),
\]
and the assertion of this lemma follows immediately.
\end{proof}

The resulting solution for $\zeta$ is illustrated on Fig.~\ref{fig:c_solution_scheme}.

\section{Slug injection solution for $\mathcal U$}
\label{sec:solution-U}

As was stated previously, when $\varphi < t_{inj}$, a slug injection solution coincides with the solution of the Riemann problem. Due to \cite{JnW}, the $\mathcal U$-solution of the Riemann problem is a combination of rarefaction waves and shocks. It is clear that 
any $\zeta$-shock is also a discontinuity in $\mathcal U$. Otherwise, if we have $\mathcal U^+ = \mathcal U^-$ on the chemical shock, then due to Sect.~\ref{sec:sec5-shock-mapping} we have $f(s^+, c^+) = f(s^-, c^-)$ in original coordinates, thus from \eqref{eq:RH-1} we obtain $v(s^+ - s^-) = 0$, which contradicts Proposition \ref{prop:inadmissible_shocks}, since both $s^+ = s^-$ and $v=0$ correspond to inadmissible shocks. Hereinafter by $\mathcal U^{\pm}_{OA}$ we denote the constant values of $\mathcal U$ above and below the straight segment of the chemical front (the segment connecting $0$ and $A$ on Fig.~\ref{fig:c_solution_scheme}) respectively. 
Due to \cite{JnW} there is a rarefaction wave that connects the value $\mathcal U^+_{OA}$ on the chemical front and the boundary value $\mathcal U = 1$ on the vertical axis between $0$ and $t_{inj}$.
A pencil of $\mathcal U$-characteristics thus starts at the origin and fills the triangle 
$$\bigtriangleup_O = \{(\varphi, x)\colon x > 0,\ \varphi > v(1, 0) x,\ \varphi < t_{inj} + a_\zeta(1) x\}.$$ 
Each of the characteristics in the triangle is a line that carries a constant value of $\mathcal U\in [1, \mathcal U^+_{OA}]$ and has the slope equal to $\mathcal F_{\mathcal U}(\mathcal U, 1)$. 
Since the lowest characteristic coincides with the chemical front shock, the corresponding value $\mathcal U^+_{OA}$ can be found using the formula 
\begin{equation}
\label{eq:U+_cf}
\mathcal F_{\mathcal U}(\mathcal U^+_{OA}, 1) = v(1, 0).
\end{equation}
Due to ($\mathcal F$3) in Proposition \ref{prop:new-flow-function-properties} the function $\mathcal F_{\mathcal U}(\cdot, 1)$ strictly decreases on $[1, \mathcal U^+_{OA}]$ and thus has an inverse finction
\[\mathcal G_\bigtriangleup = \mathcal F_{\mathcal U}^{-1}(\cdot, 1).\]
Hence, we have an expression for $\mathcal U$ in $\bigtriangleup_O$:
\begin{equation} \label{eq:U_val_pencile}
\mathcal U = \mathcal G_\bigtriangleup\left(\frac\varphi x\right),\qquad (\varphi, x)\in \bigtriangleup_O.
\end{equation}
The Rankine--Hugoniot condition \eqref{eq:RH-2} applied to the chemical shock front defines $\mathcal U_{OA}^-$ implicitly:
\[
\frac{\mathcal F(\mathcal U^+_{OA}, 1) - \mathcal F(\mathcal U^-_{OA}, 0)}{\mathcal U^+_{OA} - \mathcal U^-_{OA}} = v(1, 0).
\]
This equation has two solutions for $\mathcal U_{OA}^-$, but we choose the one satisfying the speed compatibility property for a valid self-similar Riemann problem solution:
\begin{equation}
\label{inq:U_Lax_cond}
    \mathcal U_{OA}^+ < \mathcal U_{OA}^-.
\end{equation}

Below this segment of the shock there are parallel straight characteristics with the slope $\mathcal F_{\mathcal U}(\mathcal U^-_{OA}, 0)$. 
We have to consider two cases: whether these characteristics arrive at the horizontal axis or not (and thus if an additional rarefaction wave above the $x$-axis exist or not). It depends on the sign of $\mathcal F_{\mathcal U}(\mathcal U^{-}_{OA},0)$ or equivalently due to  ($\mathcal F 2$) of Proposition \ref{prop:new-flow-function-properties} the sign of $\mathcal U^-_{OA} - \mathcal U^{\max}(0)$.
If $\mathcal U^-_{OA} < \mathcal U^{\max}(0)$,
the parallel characteristics carry the value $\mathcal U^-_{OA}$ along them from the straight segment of the chemical front shock to the $x$-axis.
If $\mathcal U^-_{OA} > \mathcal U^{\max}(0)$, there is a rarefaction wave formed by $\mathcal U$-characteristics starting at the origin with slopes in $[0, \mathcal F_{\mathcal U}(\mathcal U^-, 0)]$. Along each characteristic of this pencil the quantity $\mathcal U\in [\mathcal U^-_{OA}, \mathcal U^{\max}(0)]$ is constant. 
The case $\mathcal U^-_{OA} = \mathcal U^{\max}(0)$ corresponds to a degenerate rarefaction wave.

Let's prove that the parallel $\mathcal U$-characteristics below the the chemical front shock starting from the straight segment do not intersect the curve $\Phi$.
Substituting the coordinates of the point $A$, given by \eqref{eq:x_A_phi_A}, into the right side of the equation \eqref{eq:Phi_equation} and using 
\[
\frac{\varphi_A - t_{inj}}{x_A} = a_\zeta(1),
\]
we obtain an equality between the slope of $\Phi$ at $A$ and the slope of the straight part of the chemical front shock:
\begin{equation}
\label{eq:dPhi_dx_at_A}
\frac{d\Phi}{d x}(x_A) = v(1, 0).
\end{equation}
Due to the chain rule the second derivative of $\Phi$ with respect to $x$ is as follows
\begin{equation*}
    \frac{d^2\Phi(x)}{d x^2} = 
    x^{-1}\left.\frac{d}{d\zeta}\frac{a(\zeta)}\zeta\right|_{\zeta = g\left(\frac{\Phi(x) - t_{inj}}x\right)}
    g'\left(\frac{\Phi(x) - t_{inj}}x\right)\left(\frac{d\Phi(x)}{dx} - \frac{\Phi(x) - t_{inj}}x\right).
\end{equation*}
Let us show that each multiplier of the right hand side 
has a sign.
We have
\begin{equation}
\label{eq:shock_front_deriv_zeta}
    \frac{d}{d\zeta} \frac{a(\zeta)}\zeta = \frac{a_\zeta(\zeta)}{\zeta} - \frac{a(\zeta)}{\zeta^2}.
\end{equation}
Properties (A1) and (A3) guarantee that 
\begin{equation}
\label{inq:ads_conv}
    \frac{a(\zeta)}\zeta -  a_\zeta(\zeta) > 0
\end{equation}
 for all $\zeta > 0$.
Hence, for all $\zeta \in (0, 1]$ the following inequality holds:
\begin{equation}
\label{inq:ads_conv2}
\frac{d}{d\zeta}\frac{a(\zeta)}\zeta < 0.
\end{equation}
Property (A3) and the definition of $g$ result in the inequality
\begin{equation}
\label{ineq:g_monoton}
g'(r) < 0\quad \text{for all}\quad r \in [a_\zeta(1), a_\zeta(0)].
\end{equation}
Now all we need is to prove that all $\zeta$-characteristics intersect with $\Phi$ transversely, that is for all $x\geq x_A$ the following inequality holds true
\begin{equation}
\label{ineq:transverse_Phi}
    \frac{d\Phi(x)}{dx} > \frac{\Phi(x) - t_{inj}}x.
\end{equation}
Let us reintroduce the notation $\zeta_\Phi(x) = g\left(\frac{\Phi(x) - t_{inj}}{x}\right)$ from Lemma \ref{lemma_explicit_Phi}.
Using \eqref{eq:Phi_equation}  and  \eqref{inq:ads_conv}, we obtain 
\[
\frac{d\Phi(x)}{dx} = \frac{a(\zeta_\Phi(x))}{\zeta_\Phi(x)} > a_\zeta(\zeta_\Phi(x)) = \frac{\Phi(x) - t_{inj}}x,
\]
which gives us the desired inequality \eqref{ineq:transverse_Phi}.
Finally, combining inequalities \eqref{inq:ads_conv2}, \eqref{ineq:g_monoton} and  \eqref{ineq:transverse_Phi}, we arrive at
\begin{equation*}
    \frac{d^2\Phi(x)}{d x^2} > 0.
\end{equation*}
Hence, the slope of $\Phi(x)$ is greater than $v(1, 0)$ at every $x \geq x_A$.

Due to ($\mathcal F$3) of Proposition \ref{prop:new-flow-function-properties} the function $\mathcal F(\cdot, 0)$ is concave on $[1, \mathcal U^{\max}(0)]$ and $\mathcal F_{\mathcal U}(\mathcal U_{OA}^-,0) < 0$ for $\mathcal U>\mathcal U^{\max}(0)$. Thus, using \eqref{eq:U+_cf} and \eqref{inq:U_Lax_cond} we derive
the following inequality between the slope of the straight segment of the chemical front shock and the slope of the parallel $\mathcal U$-characteristics starting from this segment
\[
\mathcal F_{\mathcal U}(\mathcal U_{OA}^-,0) < \mathcal F_{\mathcal U}(\mathcal U_{OA}^+,1) = v(1,0).
\]
Therefore, the slope of the parallel characteristics is less than $v(1, 0)$, and thus, the parallel characteristics go below the curve $\Phi$ and do not intersect it.

\subsection{Characteristics in the cone and Joguet condition}

After applying the chain rule to the second term of \eqref{eq:U-lagr-eqn}
we obtain
\begin{equation}
    \label{eq:U-lagr-eqn2}
    \mathcal U_x + \mathcal F_{\mathcal U}(\mathcal U, \zeta)\mathcal U_\varphi + \mathcal F_\zeta(\mathcal U, \zeta) \zeta_\varphi = 0.
\end{equation}
If we consider a characteristic $\varphi = \varphi(x)$ with the slope $\mathcal{F}_\mathcal{U}(\mathcal{U},\zeta)$, we find that the two first terms in \eqref{eq:U-lagr-eqn2} form the differential 
\[
\frac{d}{dx}\mathcal{U}(\varphi(x),x) = \mathcal U_x(\varphi(x),x) + \mathcal F_{\mathcal U}(\mathcal U(\varphi(x),x), \zeta(\varphi(x),x))\mathcal U_\varphi(\varphi(x),x).
\]
Thus, the system of equations describing $\mathcal U$-characteristics $\varphi$ is as follows:
\begin{align} 
    \label{eq:U-char-phi}
    &\dfrac{d}{dx}\varphi(x) = \mathcal{F}_\mathcal{U}(\mathcal{U}(\varphi(x),x),\zeta(\varphi(x),x)),\\
    \label{eq:U-char-U}
    &\frac{d}{dx}\mathcal{U}(\varphi(x),x)= 
    -\mathcal F_\zeta\big(\mathcal{U}(\varphi(x),x),\zeta(\varphi(x),x)\big)
    \zeta_\varphi(\varphi(x),x).
\end{align}

Consider the solution $\zeta$ obtained in Section \ref{sec:c-sol}.
Substituting it into the system~\eqref{eq:U-char-phi}-\eqref{eq:U-char-U}, we want to construct $\mathcal U$-characteristics 
above the chemical shock front. 
A significant portion of that area can be covered with the extensions of the straight characteristics we constructed inside $\bigtriangleup_O$. However, as will become clear later, there is always an area with sufficiently large $x$ that such characteristics do not cover.
Therefore, we need some additional condition to be imposed on $\mathcal U$ somewhere in that area in order to construct the full solution. The proposed condition is dictated by the Jouguet principle, which states that at the chemical shock front the slope of the characteristic given by \eqref{eq:U-char-phi} should be equal to the slope of the shock front $\Phi$ given by \eqref{eq:Phi_equation}, \eqref{eq:zeta_in_cone_definition}.
\begin{definition}
The initial value $\mathcal U$ on $\Phi$ obtained via the \emph{Joguet condition} satisfies the equation
\begin{equation}
\label{eq:Joguet_cond} 
\mathcal F_{\mathcal U}(\mathcal U(\Phi(x), x), \zeta(\Phi(x), x)) = \frac{a(\zeta(\Phi(x), x))}{\zeta(\Phi(x), x)}.
\end{equation} 
\end{definition}

\begin{lemma}
\label{lemma:transverse_intersect}
Consider the cone 
\begin{equation*}
\angle_{t_{inj}} = \{(\varphi, x)\colon a_\zeta(1)x + t_{inj} \leqslant \varphi \leqslant a_\zeta(0)x + t_{inj}; \varphi \geqslant \Phi(x) \text{ for } x>x_A\}.
\end{equation*}
The $\mathcal U$-characteristics in $\angle_{t_{inj}}$ (either extended from $\bigtriangleup_O$ up or constructed via the Jouguet condition) intersect $\zeta$-characteristics transversely.
\end{lemma}
\begin{proof}
Let $\varphi(x)$ be a $\mathcal U$-cha\-rac\-te\-ris\-tic that starts transversely at the lower bound of $\angle_{t_{inj}}$, but touches a $\zeta$-characteristic at some point $x_0$ for the first time. That means 
\begin{equation}\label{eq:varphi_x_touches_c_char}
\varphi_x(x_0) = \dfrac{\varphi(x_0) - t_{inj}}{x_0}.
\end{equation}
Due to \eqref{eq:zeta_derivatives} we have
\[
\dfrac{\varphi(x_0) - t_{inj}}{x_0} = -\dfrac{\zeta_x(\varphi(x_0), x_0)}{\zeta_\varphi(\varphi(x_0), x_0)},
\]
and therefore we trivially have
\[
\left.\dfrac{d}{dx} \zeta(\varphi(x), x) \right|_{x=x_0} = \zeta_x(\varphi(x_0), x_0) + \zeta_\varphi(\varphi(x_0), x_0)\varphi_x(x_0) = 0.
\]
Now, if we consider
\begin{align*}
\varphi_{xx}(x_0) = \left.\dfrac{d}{dx} \mathcal F_{\mathcal U}(\mathcal U, \zeta) \right|_{x=x_0} &=
\left. \mathcal F_{\mathcal U\mathcal U}(\mathcal U, \zeta) \dfrac{d}{dx} \mathcal U\right|_{x=x_0} + \left.\mathcal F_{\mathcal U\zeta}(\mathcal U, \zeta) \dfrac{d}{dx} \zeta \right|_{x=x_0}\\
{} &= - \left. \mathcal F_{\mathcal U\mathcal U}(\mathcal U, \zeta) \mathcal F_\zeta(\mathcal U, \zeta) \zeta_\varphi \right|_{x=x_0}, 
\end{align*}
then it is clear that the expression on the right-hand side has a sign. From \eqref{eq:zeta_derivatives} we have $\zeta_\varphi < 0$, since $g' = 1/a_{\zeta\zeta}(g) < 0$ due to (A3). From ($\mathcal{F}$4.1) we have $\mathcal F_\zeta(\mathcal U, \zeta) < 0$. 
Due to \eqref{eq:varphi_x_touches_c_char} and \eqref{eq:U-char-phi} we have $\mathcal F_{\mathcal U}(\mathcal U, \zeta) > 0$ at $x = x_0$, therefore, due to ($\mathcal{F} 3$) we have $\mathcal F_{\mathcal U\mathcal U}(\mathcal U, \zeta) < 0$ at $x=x_0$. 
Thus, we arrive at $\varphi_{xx}(x_0) > 0$, which is impossible for a $\mathcal U$-characteristic touching a $\zeta$-characteristic (which is a straight line) from below. Therefore, $\varphi(x)$ cannot touch any $\zeta$-characteristics and must cross them transversely.
\end{proof}

\begin{corollary}
\label{corollary:F-deriv-est-in-cone}
On $\mathcal U$-characteristics in the cone $\angle_{t_{inj}}$ we have
\[
\mathcal F_{\mathcal U} (\mathcal U, \zeta) > a_\zeta(\zeta), \qquad \mathcal F_{\mathcal U \mathcal U}(\mathcal U, \zeta) < 0.
\]
\end{corollary}
\begin{proof}
Since \eqref{eq:varphi_x_touches_c_char} is impossible, from \eqref{eq:g_definition}, \eqref{eq:zeta_in_cone_definition} and \eqref{eq:U-char-phi} we have
\[
\mathcal F_{\mathcal U} (\mathcal U, \zeta) = \varphi_x(x) > \dfrac{\varphi(x) - t_{inj}}{x} = a_\zeta(\zeta).
\]
The second inequality follows from the first and ($\mathcal{F}3$).
\end{proof}

It became clear in the course of our research that the Jouguet condition cannot be always applicable in all cases for all points $(\Phi(x), x)$ on the chemical shock front. Two problems limit the applicability of the condotion:
\begin{enumerate}
    \item Since the Jouguet characteristics by definition are tangential to the front $\Phi$, it is not clear whether they will go immediately above or below the front. In the cases where they go below, such characteristics cannot be used to define the solution above the shock front. The second derivatives of the curves need to be compared in order to determine which is the case at every point.
    \item Another characteristic can arrive at $(\Phi(x), x)$ from the left and bring a value of $\mathcal U$ with it that contradicts the Jouguet condition \eqref{eq:Joguet_cond}. At such points no additional condition is necessary.
\end{enumerate}

In this subsection we formulate a novel condition derived from the second derivative comparison. Since the shape of the $\mathcal U$-characteristics is mostly driven by the function $\mathcal F$, and the shape of $\Phi$ is defined solely through the function $a$, this new condition shows an interplay of the hydrodynamic properties of the Buckley--Leverett flow and the adsorption curve in the construction of the $\mathcal U$-solution (in particular in calculating the saturation behind the chemical shock front).
We prove that this condition is directly tied to the first problem described above. In the later subsections we show that at least in some cases the second problem is also tied to the same condition. We conjecture that it might be true in all cases.

\begin{definition}
Denote 
\[
b(\zeta) = \dfrac{a(\zeta)}{\zeta} - a_\zeta(\zeta).
\]
Let $\zeta\in(0, 1]$ and $\mathcal U$ be such that $\mathcal F_{\mathcal U} (\mathcal U, \zeta) = \dfrac{a(\zeta)}{\zeta}$. We introduce the condition
\begin{itemize}
    \item[($\mathcal F$5)] $-\mathcal F_{\mathcal U \mathcal U}(\mathcal U, \zeta)\mathcal F_{\zeta}(\mathcal U, \zeta) + \left(\mathcal F_{\mathcal U \zeta}(\mathcal U, \zeta) + \dfrac{b(\zeta)}{\zeta}\right) b(\zeta) < 0$.
\end{itemize}
It either holds or breaks for each $\zeta$ in the interval.
\end{definition}

\begin{lemma}\label{lem:Joguet_char_above_shock} 
Under the condition ($\mathcal{F}$5) the characteristic started from the chemical shock front and constructed via the Jouguet principle does not go below the shock in some vicinity of the initial point.
\end{lemma}
\begin{proof}
Consider the $\mathcal U$-characteristics $\varphi(x)$ constructed via the Jouguet condition at the starting point $x_0$. At $(\Phi(x_0), x_0)$ we have
\[
\varphi_x(x_0) = \mathcal F_{\mathcal U}(\mathcal U(\Phi(x_0),x_0), \zeta(\Phi(x_0),x_0)) = {\Phi}_x(x_0) = \dfrac{a(\zeta(\Phi(x_0), x_0))}{\zeta(\Phi(x_0), x_0)}.\]
We have a strict sign (non-equal) in ($\mathcal F$5) at $(\Phi(x_0), x_0)$, then $\varphi$ is guaranteed to go above $\Phi$, since it implies
\[
\dfrac{d}{d\zeta}(\varphi_x - \Phi_x) = -\dfrac{\mathcal F_{\mathcal U \mathcal U}(\mathcal U, \zeta)\mathcal F_{\zeta}(\mathcal U, \zeta)}{\mathcal F_{\mathcal U}(\mathcal U, \zeta) - a_\zeta(\zeta)} + \mathcal F_{\mathcal U \zeta}(\mathcal U, \zeta) - \dfrac{a_\zeta(\zeta)}{\zeta} + \dfrac{a(\zeta)}{\zeta^2} < 0.
\]
Since $\zeta$ decreases the curves diverge in the upward direction in a vicinity of $(\Phi(x_0), x_0)$.
\end{proof}

\subsection{New splitting technique for characteristic system}
In this section we describe a change of variables in the cone $\angle_{t_{inj}}$. The new coordinates are similar to polar coordinates centered at $T$, with $\zeta$ as an angular coordinate and a new variable
\begin{equation}
\label{eq:psi_def}
\psi = \frac12\ln((\varphi-t_{inj})^2 + x^2)
\end{equation}
playing the role of the radial coordinate.
Mapping $(\varphi,x)\to(\psi,\zeta)$ turns the cone into a strip.

Now let us trace how $\zeta$ changes along the characteristic $\varphi(x)$.
Note that from \eqref{eq:zeta_in_cone_definition} we have
\begin{equation}\label{eq:zeta_derivatives}
\zeta_\varphi =\dfrac{1}{x} g'\left(\dfrac{\varphi - t_{inj}}{x}\right), \quad  \zeta_x = -\dfrac{\varphi - t_{inj}}{x^2} g'\left(\dfrac{\varphi - t_{inj}}{x}\right) = - \dfrac{\varphi - t_{inj}}{x} \zeta_\varphi,
\end{equation}
therefore, from \eqref{eq:zeta_in_cone_definition}, \eqref{eq:g_definition} and \eqref{eq:U-char-phi} we derive
\begin{equation}
\label{eq:d_zeta_d_x}
\begin{aligned}
\dfrac{d}{dx}\zeta(\varphi(x), x) & = \zeta_x(\varphi(x), x) + \zeta_\varphi(\varphi(x), x) \dfrac{d}{dx}\varphi(x) \\
{} & = \zeta_\varphi(\varphi(x), x) \Big(\mathcal F_{\mathcal U}(\mathcal U(\varphi(x), x), \zeta(\varphi(x), x)) - a_\zeta(\zeta(\varphi(x), x)) \Big).
\end{aligned}
\end{equation}
Substituting $\zeta_\varphi$ from this relation into \eqref{eq:U-char-U} we obtain the following simplified equation:
\begin{equation}
\label{eq:U-char-of-c}
\dfrac{d}{d\zeta} \mathcal U = -\dfrac{\mathcal F_{\zeta}(\mathcal U, \zeta)}{\mathcal F_{\mathcal U}(\mathcal U, \zeta) - a_\zeta(\zeta)}.  
\end{equation}
For $\mathcal U = 1$ we extend the right-hand side by zero, therefore allowing the solution $\mathcal U(\zeta)\equiv 1$. Thus there is a mapping of $\mathcal U$-characteristics in $(\varphi, x)$ coordinates into the trajectories of \eqref{eq:U-char-of-c} on the~$(\zeta, \mathcal U)$ plane (see Fig.~\ref{fig:full_Jouguet_U_zeta}, \ref{fig:one_sign_change_U_zeta} in Sect.~\ref{sec:Examples}). The constant solution $\mathcal U\equiv 1$ then corresponds to the degenerate one-point characteristic $\mathcal U(t_{inj}, 0) = 1$.
Due to ($\mathcal F$4) the right-hand side of \eqref{eq:U-char-of-c} is positive, therefore all integral curves are increasing functions of $\zeta$. 

Consider the implicit function $\mathcal U_{J}(\zeta)$ given by the equation
\begin{equation}
    \label{eq:Joguet_in_zeta_U}
\mathcal F_{\mathcal U}(\mathcal U_{J}(\zeta), \zeta) = \frac{a(\zeta)}\zeta, \qquad \zeta\in(0,1].
\end{equation}
We also define $\mathcal U_J(0)$ by continuity, so it could be found as the solution of 
$$\mathcal F_{\mathcal U}(\mathcal U_{J}(0), 0) = a_\zeta(0).$$
From \eqref{eq:Joguet_in_zeta_U} via implicit function derivative formula we calculate
\begin{equation}
\label{eq:d-dzeta-U-Phi}
\dfrac{d}{d\zeta} \mathcal U_J = -\dfrac{\mathcal F_{\mathcal U \zeta}(\mathcal U_J(\zeta), \zeta) + \dfrac{b(\zeta)}{\zeta}}{\mathcal F_{\mathcal U\mathcal U}(\mathcal U_J(\zeta), \zeta)}.
\end{equation}
Note that equation \eqref{eq:Joguet_in_zeta_U} reflects the Jouguet condition \eqref{eq:Joguet_cond} on the $(\zeta, \mathcal U)$ plane. The graph of the function $\mathcal U_J$ denoted by 
$$\Upsilon_J = \{(\zeta, \mathcal U): \zeta\in[0,1], \mathcal U = \mathcal U_J(\zeta)\}$$ 
corresponds to the chemical shock front $\varphi = \Phi(x)$ if and only if the characteristic at the corresponding point is constructed via the Jouguet condition. Each $\mathcal U$-characteristic beginning at $TA$ corresponds to the integral curve given by the equation \eqref{eq:U-char-of-c} and ending at the segment 
$$\Upsilon_{TA} = \{(\zeta, \mathcal U)\colon \zeta = 1,\ 1<\mathcal U \leq \mathcal U_J(1) = \mathcal U_{OA}^+\}.$$ 
Each $\mathcal U$-characteristic beginning at the point $(\Phi(x), x)$ and constructed via the Jouguet condition corresponds to the integral curve given by the same equation \eqref{eq:U-char-of-c} that arrives at $\Upsilon_J$.

Since
\begin{equation}
\label{eq:a_zeta}
a_\zeta(\zeta) = \frac{\varphi - t_{inj}}x,
\end{equation}
using \eqref{eq:U-char-phi} we obtain
\begin{equation}
\label{eq:psi-eqn}
\frac{d\psi(x)}{dx} = \frac{1 + a_\zeta(\zeta)\mathcal F_{\mathcal U}(\mathcal U, \zeta)}{e^\psi\sqrt{a_\zeta^2(\zeta)+1}}.
\end{equation}
Due to \eqref{eq:d_zeta_d_x} using \eqref{eq:zeta_derivatives} we derive 
\begin{equation}
\label{eq:d_psi_d_zeta}
    \frac{d\psi}{d\zeta} = \frac{a_{\zeta\zeta}(\zeta)a_\zeta(\zeta)}{1 + a_\zeta^2(\zeta)} + \frac{a_{\zeta\zeta}(\zeta)}{\mathcal F_{\mathcal U}(\mathcal U, \zeta) - a_\zeta(\zeta)}
\end{equation}

Suppose $(\mathcal F 5)$ holds for $\zeta(\Phi(x), x)$ at some point $(\Phi(x), x)$. Then we can map the point $(x, \Phi(x), \mathcal{U}(\Phi(x), x)))$ into the point $(\zeta, \mathcal U_J(\zeta), \psi_\Phi(\zeta))$, where $\zeta$ is calculated via \eqref{eq:zeta_in_cone_definition}, $\mathcal U_J(\zeta)$ is calculated implicitly from the equation~\eqref{eq:Joguet_in_zeta_U}, and using the notations from Lemma~\ref{lemma_explicit_Phi} and \eqref{eq:a_zeta} we calculate
\begin{align}
\label{eq:psi_Phi_definition}
\begin{split}
\psi_\Phi(\zeta) &= \frac12\ln((\Phi(x) - t_{inj})^2 + x^2) = \ln x + \dfrac{1}{2}\ln(1+ a^2_\zeta(\zeta_\Phi(x))) \\
{} &= \ln t_{inj} - \ln p(\zeta) + \dfrac{1}{2}\ln(1+ a^2_\zeta(\zeta)).
\end{split}
\end{align}
This relation also constitutes the initial conditions for the equation \eqref{eq:d_psi_d_zeta}, which, given a solution $\mathcal U(\zeta)$ provides us with the values of $\psi$ along it.

Similarly, straightforward computations give us the initial condition for $\psi$ at $\Upsilon_{TA}$. Knowing the value of $\mathcal U(\varphi, x)$ at some point $(\varphi, x)$ of $TA$ from
\[
\varphi = \mathcal F_{\mathcal U}(\mathcal U, 1) x, \quad \varphi = t_{inj} + a_\zeta(1) x
\]
we easily calculate
\[
x = \dfrac{t_{inj}}{\mathcal F_{\mathcal U}(\mathcal U, 1) - a_\zeta(1)}, \quad \dfrac{\varphi - t_{inj}}{x} = a_\zeta(1),
\]
and therefore
\begin{equation}
\label{eq:definition_of_psi_TA}
\psi_{TA}(\mathcal U) = \ln\left(\frac{t_{inj}}{\mathcal F_{\mathcal U}(\mathcal U, 1)  - a_\zeta(1)}\right) + \frac12\ln(1 + a_\zeta^2(1)).
\end{equation}

Thus, for every characteristic (eiter from $TA$ or constructed via the Jouguet condition) we can solve \eqref{eq:U-char-of-c} for $\mathcal U(\zeta)$, then substitute this solution into \eqref{eq:d_psi_d_zeta}, which from initial conditions provided above gives us the solution $\psi(\zeta; \mathcal U(\cdot))$. 

\begin{definition}
\label{def:mapping_and_solution_shorthands}
For all $\mathcal U$-characteristics in $\angle_{t_{inj}}$ we define the mapping
\begin{equation}
\label{eq:mapping_into_zeta_U}
(x, \varphi(x), \mathcal U(\varphi(x), x)) \to (\zeta, \mathcal U(\zeta), \psi(\zeta; \mathcal U(\cdot))).
\end{equation}
Let $\mathcal U(\zeta;\zeta_0, \mathcal U_0)$ be the solution of~\eqref{eq:U-char-of-c} with initial condition $\mathcal U(\zeta_0) = \mathcal U_0$. For brevity we will use the shorthands $\mathcal U(\zeta;\zeta_0) = \mathcal U(\zeta;\zeta_0, \mathcal U_J(\zeta_0))$ and $\mathcal U(\zeta;\mathcal U_0) = \mathcal U(\zeta;1, \mathcal U_0)$ when appropriate. Similarly, given the solution $\mathcal U(\zeta;\zeta_0)$, we shorthand $\psi(\zeta; \zeta_0) = \psi(\zeta; \mathcal U(\cdot;\zeta_0))$ and for $\mathcal U(\zeta;\mathcal U_0)$, we write $\psi(\zeta;\mathcal U_0) = \psi(\zeta; \mathcal U(\cdot;\mathcal U_0))$.
\end{definition}

\begin{proposition}[Smooth dependence on parameters]
\label{prop:continuous_differential_by_parameter}
The functions $\mathcal U(\zeta;\zeta_0)$, $\mathcal U(\zeta;\mathcal U_0)$, $\psi(\zeta; \zeta_0)$ and $\psi(\zeta;\mathcal U_0)$ are continuously differentiable with respect to their parameters.
\end{proposition}

\begin{lemma}
\label{lemma_dU_dzeta0}
Let ($\mathcal F$5) hold for $\zeta_0$. Then $\dfrac{\partial}{\partial {\zeta_0}}\mathcal U(\zeta; \zeta_0) > 0$ for $\zeta, \zeta_0 \in [0,1]$.
\end{lemma}
\begin{proof}
Changing the order of differentiation we obtain
    \begin{equation*}
\frac{d}{d\zeta}\left(\mathcal U_{\zeta_0}(\zeta; \zeta_0)\right) = -\frac{\partial}{\partial\mathcal U}\left(\dfrac{\mathcal F_{\zeta}(\mathcal U(\zeta; \zeta_0), \zeta)}{\mathcal F_{\mathcal U}(\mathcal U(\zeta; \zeta_0), \zeta) - a_\zeta(\zeta)}\right)\mathcal U_{\zeta_0}(\zeta).  
\end{equation*}
Integrating this we get
\begin{equation*}
    \mathcal U_{\zeta_0}(\zeta; \zeta_0) = \mathcal U_{\zeta_0}(\zeta_0; \zeta_0) \exp\left(-\int\limits_{\zeta_0}^\zeta \frac{\partial}{\partial\mathcal U}\left(\dfrac{\mathcal F_{\zeta}(\mathcal U(s; \zeta_0), s)}{\mathcal F_{\mathcal U}(\mathcal U(s; \zeta_0), s) - a_\zeta(s)}\right)ds\right).
\end{equation*}
Recall that due to the boundary condition we have
\[
\mathcal U(\zeta_0, \zeta_0) = \mathcal U_J(\zeta_0).
\]
Taking the derivative with respect to $\zeta_0$ we obtain
\[
\mathcal U_{\zeta_0}(\zeta_0; \zeta_0) + \mathcal U_{\zeta}(\zeta_0; \zeta_0) = \mathcal U'_J(\zeta_0).
\]
From \eqref{eq:U-char-of-c} and \eqref{eq:d-dzeta-U-Phi} via the condition ($\mathcal F$5) we obtain 
\[
\mathcal U_{\zeta_0}(\zeta_0; \zeta_0) > 0,
\]
and the lemma is proved.
\end{proof}

\begin{lemma}
\label{lemma_U_derivative_by_U_0}
$\dfrac{\partial}{\partial \mathcal U_0}\mathcal U(\zeta; \mathcal U_0) > 0$ for $\zeta\in[0,1]$, $\mathcal U_0\in[1, \mathcal U^+_{OA}]$.
\end{lemma}
\begin{proof}
Repeating the scheme of proof of Lemma~\ref{lemma_dU_dzeta0}, we obtain
\begin{equation*}
    \mathcal U_{\mathcal U_0}(\zeta; \mathcal U_0) = \mathcal U_{\mathcal U_0}(1; \mathcal U_0) \exp\left(-\int\limits_{1}^\zeta \frac{\partial}{\partial\mathcal U}\left(\dfrac{\mathcal F_{\zeta}(\mathcal U(s; \mathcal U_0), s)}{\mathcal F_{\mathcal U}(\mathcal U(s; \mathcal U_0), s) - a_\zeta(s)}\right)ds\right).
\end{equation*}
It is clear that $\mathcal U_{\mathcal U_0}(1; \mathcal U_0) = 1$, and the exponent is positive, therefore the lemma is proved.
\end{proof}

\begin{lemma} 
\label{lemma:dpsi_dzeta0}
Let ($\mathcal F$5) hold for $\zeta_0$. Then 
$\dfrac{\partial}{\partial\zeta_0}\psi(\zeta;\zeta_0) < 0$ for $\zeta, \zeta_0 \in [0,1]$, $\zeta < \zeta_0$.
\end{lemma}
\begin{proof}
First, we rewrite \eqref{eq:d_psi_d_zeta} in the integral form
    \[
    \psi(\zeta;\zeta_0) = \psi_\Phi(\zeta_0) + \int\limits_{\zeta_0}^\zeta\frac{a_{\zeta\zeta}(s) a_{\zeta}(s)}{1+ a_\zeta^2(s)} +
    \frac{a_{\zeta\zeta}(s)}{\mathcal F_{\mathcal U}(\mathcal U(s;\zeta_0), s) - a_\zeta(s)} ds,
    \]
and differentiate it with respect to $\zeta_0$:
    \[
    \frac{\partial \psi}{\partial\zeta_0}(\zeta;\zeta_0) = \psi'_\Phi(\zeta_0) - \psi_\zeta(\zeta_0;\zeta_0) + \int\limits_{\zeta_0}^\zeta \frac{\partial}{\partial \zeta_0}\left(\frac{a_{\zeta\zeta}(s)}{\mathcal F_{\mathcal U}(\mathcal U(s;\zeta_0), s) - a_\zeta(s)}\right) ds.
    \]
Direct calculations show that
    \begin{align*}
    \psi'_\Phi(\zeta) &= -\dfrac{p_\zeta(\zeta)}{p(\zeta)} + \dfrac{a_\zeta(\zeta)a_{\zeta\zeta}(\zeta)}{1 + a^2_\zeta(\zeta)} = 
        \dfrac{(\zeta + a(\zeta)a_\zeta(\zeta))a_{\zeta\zeta}(\zeta)}{(1 + a^2_\zeta(\zeta)) (a(\zeta)-\zeta a_\zeta(\zeta))}
    \end{align*}
and due to the Jouguet condition $\mathcal F_{\mathcal U}(\mathcal U(\zeta_0;\zeta_0), \zeta_0) = \dfrac{a(\zeta_0)}{\zeta_0}$ we obtain
    \begin{align*}
    \psi_\zeta(\zeta_0;\zeta_0) &= \frac{a_{\zeta\zeta}(\zeta_0) a_{\zeta}(\zeta_0)}{1+ a_\zeta^2(\zeta_0)} +
    \frac{a_{\zeta\zeta}(\zeta_0)}{\mathcal F_{\mathcal U}(\mathcal U(\zeta_0;\zeta_0), \zeta_0) - a_\zeta(\zeta_0)} \\
    {} &= \dfrac{(\zeta_0 + a(\zeta_0)a_\zeta(\zeta_0))a_{\zeta\zeta}(\zeta_0)}{(1+ a_\zeta^2(\zeta_0)) (a(\zeta_0)-\zeta_0 a_\zeta(\zeta_0))},
    \end{align*}
therefore 
\[
\psi'_\Phi(\zeta_0) -\psi_\zeta(\zeta_0;\zeta_0) = 0.
\]
Thus,
\[
\frac{\partial \psi}{\partial\zeta_0}(\zeta;\zeta_0) = 
\int\limits_{\zeta_0}^\zeta \frac{a_{\zeta\zeta}(s)}{\mathcal F_{\mathcal U}(\mathcal U(s;\zeta_0), s) - a_\zeta(s)}
\mathcal F_{\mathcal U\mathcal U}(\mathcal U(s;\zeta_0), s)\frac{\partial \mathcal U}{\partial\zeta_0} (s;\zeta_0)ds.
\]
Due to Lemma~\ref{lemma_dU_dzeta0}, Corollary~\ref{corollary:F-deriv-est-in-cone} and (A3) the integrand is positive, thus for $\zeta < \zeta_0$ the integral is negative.
Therefore, the lemma is proved.
\end{proof}

\begin{lemma}
\label{lemma:dpsi_dU0}
$\dfrac{\partial}{\partial \mathcal U_0}\psi(\zeta; \mathcal U_0) > 0$  for $\zeta\in[0,1]$, $\mathcal U_0\in[1, \mathcal U^+_{OA}]$.
\end{lemma}
\begin{proof}
This follows immediately from \eqref{eq:d_psi_d_zeta}. Let us rewrite it in integral form:
\[
\psi(\zeta; \mathcal U_0) = \psi_{TA}(\mathcal U_0) + \int\limits_{1}^\zeta \frac{a_{\zeta\zeta}(s)a_\zeta(s)}{1 + a_\zeta^2(s)} + \frac{a_{\zeta\zeta}(s)}{\mathcal F_{\mathcal U}(\mathcal U(s; \mathcal U_0), s) - a_\zeta(s)} \, ds.
\]
The integrand is negative and decreasing with respect to $\mathcal U_0$ due to Lemma~\ref{lemma_U_derivative_by_U_0} and Corollary~\ref{corollary:F-deriv-est-in-cone}. In the integration limits $\zeta \leqslant 1$, therefore, the integral is increasing. Finally, from the definition \eqref{eq:definition_of_psi_TA} we see that $\psi_{TA}$ is strictly increasing:
\[
\psi'_{TA}(\mathcal U_0) = -\dfrac{\mathcal F_{\mathcal U\mathcal U}(\mathcal U_0, 1)}{\mathcal F_{\mathcal U}(\mathcal U_0, 1) - a_\zeta(1)} > 0.
\]
\end{proof}

\subsection{Solution for $\mathcal U$ in the cone}

\subsubsection{Full Jouguet case}

In this subsection we assume ($\mathcal F$5) holds for all $\zeta$. The solution, therefore, requires the Jouguet condition at every point of the chemical shock front. We prove that these Jouguet characteristics cover the cone $\angle_{t_{inj}}$ with no gaps or intersections and give us the solution satisfying all necessary conditions for the uniqueness theorem.

First, we demonstrate that there are no gaps in $(\zeta, \mathcal U)$ plane.
\begin{proposition}
\label{prop:full_cover_of_zeta_U_plane}
Let the condition ($\mathcal F$5) hold true for $\zeta(\Phi(x), x)$ at every point $(\Phi(x), x)$  of the chemical shock front. Then the images of $\mathcal U$-characteristics in $\angle_{t_{inj}}$ (continued from TA and constructed via the Jouguet condition \eqref{eq:Joguet_cond} from each point of the chemical shock front $(\Phi(x),x)$) when mapped onto the $(\zeta, \mathcal U)$ plane fill the whole area
$$\overline {\Upsilon}_J = \{(\zeta, \mathcal U): \zeta\in[0,1], \mathcal U \in [1, \mathcal U_J(\zeta)]\}$$
under the graph of $\mathcal U_J$.
Each of $\mathcal U$-characteristic images starts at the segment $\zeta = 0$, $\mathcal U \in [0,\mathcal U_J(0)]$ and arrives either at $\Upsilon_J$ or at $\Upsilon_{TA}$.    
\end{proposition}
\begin{proof}
The condition ($\mathcal F$5) implies that all integral curves arrive at $\Upsilon_J$ from below. 
Indeed, comparing the right-hand side of \eqref{eq:d-dzeta-U-Phi} with the right-hand side of \eqref{eq:U-char-of-c}, due to ($\mathcal F$5) we see that $\frac{d}{d\zeta}\mathcal U_J < \frac{d}{d\zeta}\mathcal U$ and hence all integral curves given by \eqref{eq:U-char-of-c} intersect the graph of $\mathcal U_J$ from below. 

Let us prove that the integral curves $\mathcal U(\zeta)$ cannot start from the lower boundary, where $\mathcal U = 1$ (treating this case separately is necessary, because the right-hand side is not Lipschitz there). For this purpose we rewrite the equation \eqref{eq:U-char-of-c} in the form
\[
\mathcal F_{\mathcal U}(\mathcal U(\zeta), \zeta) \frac{d\mathcal U}{d \zeta} + \mathcal F_\zeta(\mathcal U(\zeta),\zeta) =   a_\zeta(\zeta)\frac{d\mathcal U}{d \zeta}
\]
and express the derivative
\[
 \frac{d\mathcal U}{d \zeta} = \frac{1}{a_\zeta(\zeta)}\frac{d}{d\zeta}\mathcal F(\mathcal U(\zeta), \zeta).
\]
The estimate of the adsorption derivative due to (A3) leads to an estimate
\[
\frac{d\mathcal U}{d\zeta}(\zeta) \geq \frac{1}{a_\zeta(0)}\frac{d}{d\zeta}\mathcal F(\mathcal U(\zeta), \zeta).
\]
If $\mathcal U(\zeta)$ begins at some point $(\zeta_0, \mathcal U(\zeta_0))$, where $\mathcal U(\zeta_0) = 1$, then integrating the inequality from $\zeta_0$ to $\zeta$ and keeping in mind ($\mathcal F$4) we obtain
\[
\mathcal U(\zeta) - 1\geq \frac{1}{a_\zeta(0)}\left(
\mathcal F(\mathcal U(\zeta),\zeta) + 1\right)\geq \frac{1}{a_\zeta(0)}\left(
\mathcal F(\mathcal U(\zeta),1) + 1\right),
\]
therefore for any $\mathcal U(\zeta) \neq 1$ we derive
\[
\frac{\mathcal F(\mathcal U(\zeta), 1) + 1}{\mathcal U(\zeta) - 1}\leq a_\zeta(0).
\]
Due to ($\mathcal F$2) the left-hand side tends to $+\infty$ as $\zeta\to 1$ which leads to a contradiction. Therefore our assumptions $\mathcal U(\zeta_0) = 1$, $\mathcal U(\zeta) \neq 1$ cannot hold at the same time. Thus, only the trajectory $\mathcal U(\zeta) \equiv 1$ could have the value $1$.

Due to the Picard--Lindel\"of theorem the integral curves generated by the equation \eqref{eq:U-char-of-c} fill the whole area $\overline {\Upsilon}_J$ under the graph of $\mathcal U_J$ 
since the right-hand side 
is a locally Lipschitz function in $\overline{\Upsilon}_J \setminus \{(\zeta, \mathcal U): \zeta\in[0,1], \mathcal U = 1\}$.
\end{proof}

Now we can bring this property back to the $(\varphi, x)$ plane.

\begin{theorem}
\label{thm1}
Let the condition ($\mathcal F$5) hold true for $\zeta(\Phi(x), x)$ at every point $(\Phi(x), x)$  of the chemical shock front. Then the $\mathcal U$-characteristics in $\angle_{t_{inj}}$ (continued from TA and constructed via the Jouguet condition \eqref{eq:Joguet_cond} from each point of the chemical shock front $(\Phi(x),x)$) fill the whole area $\angle_{t_{inj}}$ without intersections and give us a pice-wise continuously differentiable solution $\mathcal U(\varphi, x)$.
\end{theorem}
\begin{proof}
Note that the characteristic from the point $A$ as the continuation of the characteristic $OA$, and the characteristic from $A$ constructed via the Jouguet condition give us the same curve due to \eqref{eq:dPhi_dx_at_A}. Therefore we have two families of characteristics that have the $A$ characteristic in common.

Similar to Proposition~\ref{prop:full_cover_of_zeta_U_plane} we first map all characteristics into the $(\zeta, \mathcal U)$ plane using the mapping \eqref{eq:mapping_into_zeta_U}. In the $(\zeta, \mathcal U)$ plane the curves $\mathcal U(\zeta; \mathcal U_0)$ correspond to the characteristics continued from TA, while the curves $\mathcal U(\zeta; \zeta_0)$ correspond to the Jouguet characteristics (see Definition~\ref{def:mapping_and_solution_shorthands} for the solution shorthands we use here).

Due to Proposition~\ref{prop:full_cover_of_zeta_U_plane} together the curves $\mathcal U(\zeta; \mathcal U_0)$ and $\mathcal U(\zeta; \zeta_0)$ fill the whole area $\overline{\Upsilon}_J$.

Now, we consider the solutions $\psi(\zeta; \mathcal U_0)$ and $\psi(\zeta; \zeta_0)$ and their inverse mapping
\[
x = \frac{e^\psi}{\sqrt{a_\zeta^2(\zeta) + 1}}, \qquad \varphi = t_{inj} + \frac{e^{\psi}a_\zeta(\zeta)}{\sqrt{a_\zeta^2(\zeta) + 1}}
\]
back into the $(\varphi, x)$ plane. Due to Proposition~\ref{prop:continuous_differential_by_parameter}, Lemma~\ref{lemma:dpsi_dzeta0} and Lemma~\ref{lemma:dpsi_dU0} they are continuous and monotone with respect to their parameters. On the upper boundary $\Upsilon_J$ the functions $\psi(\zeta; \zeta_0)$ have boundary values $\psi(\zeta_0; \zeta_0) = \psi_\Phi(\zeta_0)$, so when we go back to $(\varphi, x)$ coordinates, this boundary will map back into the shock front $\Phi$. Similarly, on the right boundary of $\overline{\Upsilon}_J$ we have $\psi(1; \mathcal U_0) = \psi_{TA}(\mathcal U_0)$, which will map back into $TA$. On the lower boundary we have $\mathcal U \equiv 1$, so when approaching that boundary, $\psi(1; \mathcal U_0) \to -\infty$ as $\mathcal U_0\to 1$, and the right-hand side of \eqref{eq:d_psi_d_zeta} is bounded, so $\psi(\zeta; \mathcal U_0) \to -\infty$ as $\mathcal U_0\to 1$. Therefore, the lower boundary will map back into the point $T$. Finally, the left boundary maps into the line $\varphi = t_{inj}+a_\zeta(0)x$, since $\zeta = 0$ on it, and as $\mathcal U$ approaches $\mathcal U_J(0)$, the values of $\psi(\zeta; \zeta_0)$ must go to infinity, since $\psi_\Phi(\zeta) \to \infty$ as $\zeta\to 0$ by the definition \eqref{eq:psi_Phi_definition}.
In summary, we demonstrated that the boundaries of $\overline{\Upsilon}_J$ map back into the boundaries of $\angle_{t_{inj}}$, and since $\psi$ is continuous and monotone with respect to the parameters, the interior will also map one-to-one. Thus, the characteristics in the $(\varphi, x)$ plane cover the whole area $\angle_{t_{inj}}$.

What's left is to prove that $\mathcal U$, when mapped back into the $(\varphi, x)$ plane, gives us a piecewise continuously differentiable function. We can prove this separately for the families $\mathcal U(\zeta; \mathcal U_0)$ and $\mathcal U(\zeta; \zeta_0)$. The following proof is for the $\mathcal U(\zeta; \zeta_0)$ family, the proof for the other family is the same.

Denote by
\[
L(\zeta; \zeta_0) = \left(t_{inj} + \frac{e^{\psi(\zeta; \zeta_0)}a_\zeta(\zeta)}{\sqrt{a_\zeta^2(\zeta) + 1}}, \frac{e^{\psi(\zeta; \zeta_0)}}{\sqrt{a_\zeta^2(\zeta) + 1}}\right).
\]
Due to Proposition~\ref{prop:continuous_differential_by_parameter} this is a $\mathcal C^1$ mapping. Using this mapping we can define the inverse mapping of $\mathcal U(\zeta; \zeta_0)$ back into the $(\varphi, x)$ plane as follows:
\[
\mathcal U(\varphi, x) := \mathcal U(L^{-1}(\varphi, x)).
\]
For the inverse mapping $L^{-1}$ to exist and be continuously differential, we need its Jacobian determinant to be non-zero. Denote
\[
G(\zeta, \zeta_0) = \frac{e^{\psi(\zeta; \zeta_0)}}{\sqrt{a_\zeta^2(\zeta) + 1}}, \qquad
L(\zeta, \zeta_0) = \Big(t_{inj} + a_\zeta(\zeta)G(\zeta, \zeta_0), G(\zeta, \zeta_0)\Big).
\]
Then the Jacobian determinant is calculated as follows:
\begin{align*}
\det J_L &= (a_{\zeta\zeta}(\zeta) G(\zeta, \zeta_0) + a_{\zeta} G_\zeta(\zeta, \zeta_0)) G_{\zeta_0}(\zeta, \zeta_0) - a_\zeta(\zeta) G_{\zeta_0}(\zeta, \zeta_0) G_\zeta(\zeta, \zeta_0) \\
{} &= a_{\zeta\zeta}(\zeta) G(\zeta, \zeta_0) G_{\zeta_0}(\zeta, \zeta_0) = a_{\zeta\zeta}(\zeta) G^2(\zeta, \zeta_0) \psi_{\zeta_0}(\zeta; \zeta_0) > 0
\end{align*}
for $\zeta<\zeta_0$ due to Lemma~\ref{lemma:dpsi_dzeta0}.
\end{proof}

\subsubsection{Partial Jouguet case with one change of sign}

In this subsection we consider the simple case where the left-hand side in the condition $(\mathcal F 5)$ changes sign exactly once. In this case part of the chemical shock front does not need the Jouguet condition. 

Denote $B = (\varphi_B = \Phi(x_B), x_B)$ the unique point such that ($\mathcal F$5) changes sign at $\zeta_B = \zeta(\varphi_B, x_B)$. It is clear that at this point ($\mathcal F$5) achieves equality. But we can also tell its signs to the right and to the left with the following Lemma.

\begin{lemma}
The condition $(\mathcal F 5)$ always holds for sufficiently small $\zeta$.
\end{lemma}
\begin{proof}
Recall that in $(\mathcal F 5)$ we have $\mathcal U = \mathcal U_J(\zeta)$ defined by the Jouguet condition~\eqref{eq:Joguet_in_zeta_U}, thus it is separated from $\mathcal U = 1$. As $\zeta \to 0$, we have $b(\zeta) \to 0$, 
therefore,
\[
-\mathcal F_{\mathcal U \mathcal U}(\mathcal U, \zeta)\mathcal F_{\zeta}(\mathcal U, \zeta) + \left(\mathcal F_{\mathcal U \zeta}(\mathcal U, \zeta) + \dfrac{b(\zeta)}{\zeta}\right) b(\zeta) \to -\mathcal F_{\mathcal U \mathcal U}(\mathcal U_J(0), 0)\mathcal F_{\zeta}(\mathcal U_J(0), 0) < 0,
\]
and the condition $(\mathcal F 5)$ must hold for sufficiently small $\zeta$.
\end{proof}

From this lemma we conclude that $(\mathcal F 5)$ holds for $\zeta < \zeta_B$ (and thus for points $(\Phi(x), x)$ with $x>x_B$) and breaks for $\zeta > \zeta_B$ (for points $(\Phi(x), x)$ with $x<x_B$).

We expect that after $B$ the Jouguet condition becomes necessary, but before it the values of $\mathcal U$ need to come from characteristics continued from $TA$. Let us see what part of $\angle_{t_{inj}}$ these characteristics actually cover.

\begin{lemma}
\label{lemma_in_4_parts}
The following properties hold for the $\mathcal U$-characteristics from $TA$:
\begin{enumerate}
    \item For sufficiently small $x$ the characteristic from the point $(\varphi, x) \in TA$ does not intersect the chemical shock front $\Phi$.
    \item The characteristics in the vicinity of $x_A$ intersect the chemical shock front $\Phi$.
    \item There exists a unique point $C = (\varphi_C, x_C)\in TA$, such that the characteristic from $C$ touches the chemical shock front $\Phi$.
    \item The characteristic from $C$ touches the chemical shock front $\Phi$ exactly at the point $B$.
\end{enumerate}
\end{lemma}
\begin{proof}
1. If we move to the $(\zeta, \mathcal U)$ plane, the curve $\Phi$ corresponds to $\psi_\Phi(\zeta) \geqslant \psi_\Phi(1) > -\infty$. But as $\mathcal U_0 \to 1$, the solution 
\[
\psi(\zeta; \mathcal U_0) = \psi_{TA}(\mathcal U_0) + \int\limits_{1}^\zeta \frac{a_{\zeta\zeta}(s)a_\zeta(s)}{1 + a_\zeta^2(s)} + \frac{a_{\zeta\zeta}(s)}{\mathcal F_{\mathcal U}(\mathcal U(s; \mathcal U_0), s) - a_\zeta(s)} \, ds
\]
must uniformly tend to $-\infty$, since $\psi_{TA}(\mathcal U_0) \to -\infty$ and the integrand is bounded. Therefore, there exists a sufficiently small $\mathcal U_0$ for which $\psi(\zeta; \mathcal U_0) < \psi_\Phi(1) \leqslant\psi_\Phi(\zeta)$. And geometrically this meant the corresponding characteristic never intersects the chemical shock front $\Phi$. 

2. At $A$ the condition $(\mathcal F 5)$ breaks and has the strict opposite sign. Suppose there are characteristics from $TA$ in any vicinity of $A$
that do not intersect $\Phi$. The characteristics depend continuously on the initial data, so as their limit, the characteristic from $A$ must touch $\Phi$ from above, which gives us
\[
\left.\dfrac{d}{d\zeta}(\varphi_A'(x) - \Phi'(x))\right|_{x=x_A} \leqslant 0,
\]
which, through analysis similar to Lemma~\ref{lem:Joguet_char_above_shock}, contradicts the strict opposite sign in~$(\mathcal F 5)$. 

3. Since the characteristics depend continuously on the initial data, between the vicinity of $T$, where characteristics do not intersect $\Phi$ and the vicinity of $A$, where characteristics intersect $\Phi$, there must exist a characteristic that touches $\Phi$. We denote by $C = (\varphi_C(x_C), x_C)\in TA$ the beginning point of such characteristic. Without loss of generality, we assume that it is the supremum of all characteristics that intersect $\Phi$, and thus has characteristics intersecting $\Phi$ in any vicinity to the right of $C$. 

4. Let's map the characteristic $\varphi_C(x)$ and some of its neighbors into the plane $(\zeta, \mathcal U)$. 
At the point $x_0$, where the characteristic from $C$ touches the shock front $\Phi$, similarly to part 2 of this lemma we have
\[
\varphi_C(x_0) = \Phi(x_0), \qquad \varphi_C'(x_0) = \Phi'(x_0), \qquad \left.\dfrac{d}{d\zeta}(\varphi_C'(x) - \Phi'(x))\right|_{x=x_0} \leqslant 0,
\]
therefore $(\mathcal F 5)$ either holds or achieves equality at this point, thus $x_0 \geqslant x_B$ and $\zeta_0 = \zeta(\Phi(x_0), x_0) \leqslant \zeta_B$. At the same time, there are characteristics $\varphi_{C+\varepsilon}(x)$ intersecting $\Phi$ in any vicinity to the right of the characteristic $\varphi_C(x)$. For such characteristics at the point of intersection $x_\varepsilon$ we have
\[
\mathcal F_{\mathcal U}(\mathcal U(\Phi(x_\varepsilon), x_\varepsilon), \zeta(\Phi(x_\varepsilon), x_\varepsilon)) = \varphi_{C+\varepsilon}'(x_\varepsilon) \leqslant \Phi'(x_\varepsilon) = \dfrac{a(\zeta(\Phi(x_\varepsilon), x_\varepsilon))}{\zeta(\Phi(x_\varepsilon), x_\varepsilon)},
\]
thus for $\zeta_\varepsilon = \zeta(\Phi(x_\varepsilon), x_\varepsilon)$ we obtain 
\[
\mathcal U_{C+\varepsilon}(\zeta_\varepsilon) \geqslant \mathcal U_J(\zeta_\varepsilon),
\]
where $\mathcal U_{C+\varepsilon}$ corresponds to $\varphi_{C+\varepsilon}$ in the $(\zeta, \mathcal U)$ plane. Therefore $\mathcal U_{C+\varepsilon}$ must intersect $\mathcal U_J$. At the point of intersection $(\mathcal F 5)$ breaks, and as $\varepsilon\to 0$, we see that $(\mathcal F 5)$ breaks in any right vicinity of $\zeta_0$. Therefore $\zeta_0 \geqslant \zeta_B$. Thus, $\zeta_0 = \zeta_B$ and $x_0 = x_B$.
\end{proof}

Now we have 3 families of characteristics: 
\begin{itemize}
    \item Characteristics starting on $TC$ and going up to the upper boundary of $\angle_{t_{inj}}$. 
    \item Characteristics starting on $CA$ and intersecting the chemical shock front $\Phi$ between $A$ and $B$.
    \item Characteristics starting from every point $(\Phi(x), x)$ for $x \geqslant x_B$ constructed via the Jouguet condition. 
\end{itemize}

The second family brings some values of $\mathcal U$ to the shock front $\Phi$ between $A$ and $B$. When we move to the $(\zeta, \mathcal U)$ plane, this values map to $\mathcal U_\Phi(\zeta) = \mathcal U(\Phi(x), x)$ for~$\zeta \in [\zeta_B, 1]$.

\begin{lemma}
\label{lemma-U-Phi-above-U-J}
$\mathcal U_\Phi(\zeta) > \mathcal U_J(\zeta)$ for $\zeta > \zeta_B$ and $\mathcal U_\Phi(\zeta_B) = \mathcal U_J(\zeta_B)$.
\end{lemma}
\begin{proof}
Similar to part 4 in Lemma~\ref{lemma_in_4_parts}, at the point of intersection $x_0$ between a characteristic $\varphi(x)$ and $\Phi(x)$ we have
\[
\mathcal F_{\mathcal U}(\mathcal U(\Phi(x_0), x_0), \zeta(\Phi(x_0), x_0)) = \varphi'(x_0) \leqslant \Phi'(x_0) = \dfrac{a(\zeta(\Phi(x_0), x_0))}{\zeta(\Phi(x_0), x_0)},
\]
therefore by definition $\mathcal U_\Phi(\zeta) \geqslant \mathcal U_J(\zeta)$, and equality is reached if and only if the characteristic touches the front.
\end{proof}

\begin{proposition}
Let the condition ($\mathcal F$5) hold true only for $\zeta < \zeta_B = \zeta(\Phi(x_B), x_B)$. Then the images of $\mathcal U$-characteristics in $\angle_{t_{inj}}$ (continued from TA and constructed via the Jouguet condition \eqref{eq:Joguet_cond} from each point of the chemical shock front $(\Phi(x),x)$ with $x > x_B$) when mapped onto the $(\zeta, \mathcal U)$ plane fill the whole area
$$\overline {\Upsilon}_\Phi = \{(\zeta, \mathcal U): \zeta\in[0,\zeta_B], \mathcal U \in [1, \mathcal U_J(\zeta)] \text{ or } \zeta\in[\zeta_B, 1], \mathcal U \in [1, \mathcal U_\Phi(\zeta)]\}.$$
\end{proposition}

\begin{proof}
Our three families of characteristics map into three families of trajectories $\mathcal U(\zeta)$ when moved onto the $(\zeta, \mathcal U)$ plane.

The characteristics beginning from $TC$ map into trajectories below the trajectory $\mathcal U_C(\zeta)$ touching the upper boundary of $\overline\Upsilon_\Phi$ at the point $\zeta_B$, which corresponds to the characteristic $\varphi_C(x)$ that starts from $C$ and touches $\Phi$ at $x_B$. These characteristics go all the way from the left border $\zeta=0$ to the right border $\zeta=1$.

The characteristics beginning from $CA$ map into partial trajectories above $\mathcal U_C(\zeta)$ bounded by the right border $\zeta=1$ on the right and the graph of $\mathcal U_\Phi(\zeta)$ on the left.

The characteristics constructed via the Jouguet condition map into partial trajectories above $\mathcal U_C(\zeta)$ bounded by the left border $\zeta=0$ on the left and the graph of $\mathcal U_J(\zeta)$ on the right.

Together these three families of trajectories fill the whole area $\overline {\Upsilon}_\Phi$ due to the Picard--Lindel\"of theorem.
\end{proof}

\begin{theorem}
Let the condition ($\mathcal F$5) hold true only for $\zeta < \zeta_B = \zeta(\Phi(x_B), x_B)$. Then the $\mathcal U$-characteristics in $\angle_{t_{inj}}$ (continued from TA and constructed via the Jouguet condition \eqref{eq:Joguet_cond} from each point of the chemical shock front $(\Phi(x),x)$ with $x>x_B$) fill the whole area $\angle_{t_{inj}}$ without intersections and give us a pice-wise continuously differentiable solution $\mathcal U(\varphi, x)$.
\end{theorem}

\begin{proof}
Going through the same scheme as Theorem~\ref{thm1},
we first map all characteristics into the $(\zeta, \mathcal U)$ plane using the mapping \eqref{eq:mapping_into_zeta_U}. In the $(\zeta, \mathcal U)$ plane the curves $\mathcal U(\zeta; \mathcal U_0)$ correspond to the characteristics continued from $TA$ (two families, one below $\mathcal U_C(\zeta)$ going across the whole area and one above it that is bounded by $\mathcal U_\Phi(\zeta)$), while the curves $\mathcal U(\zeta; \zeta_0)$, $\zeta_0\in[0,\zeta_B]$, correspond to the Jouguet characteristics.
Due to Proposition~\ref{prop:full_cover_of_zeta_U_plane} together these curves fill the whole area $\overline{\Upsilon}_\Phi$.

Now, we consider the solutions $\psi(\zeta; \mathcal U_0)$ and $\psi(\zeta; \zeta_0)$ and their inverse mapping
back into the $(\varphi, x)$ plane. Due to Proposition~\ref{prop:continuous_differential_by_parameter}, Lemma~\ref{lemma:dpsi_dzeta0} and Lemma~\ref{lemma:dpsi_dU0} they are continuous and monotone with respect to their parameters. Moreover, there is a continuous transition from $\psi(\zeta; \mathcal U_0)$ to $\psi(\zeta; \zeta_0)$ over the $\mathcal U_C(\zeta)$ curve. Indeed, $\psi(\zeta_B; \zeta_B) = \psi(\zeta_B; \mathcal U_C(1)) = \psi(\Phi(x_B), x_B)$. Therefore, $\psi(\zeta; \zeta_B) = \psi(\zeta; \mathcal U_C(1))$ for $\zeta\in[0,\zeta_B]$, and two families meet continuously on $\mathcal U_C(\zeta)$. On the upper boundary $\mathcal U_J(\zeta)$, $\zeta\in[0,\zeta_B]$, the functions $\psi(\zeta; \zeta_0)$ have boundary values $\psi(\zeta_0; \zeta_0) = \psi_\Phi(\zeta_0)$, so when we go back to $(\varphi, x)$ coordinates, this boundary will map back into the shock front $\Phi(x)$ for $x\geqslant x_B$. On the upper boundary $\mathcal U_\Phi(\zeta)$, $\zeta\in[\zeta_B,1]$, the functions $\psi(\zeta; \mathcal U_0)$ have values $\psi(\zeta; \mathcal U_0) = \psi_\Phi(\zeta)$ for $\zeta$ such that $\mathcal U(\zeta; \mathcal U_0) = \mathcal U_\Phi(\zeta)$ due to the definition of~$\mathcal U_\Phi(\zeta)$. On the right boundary of $\overline{\Upsilon}_\Phi$ we have $\psi(1; \mathcal U_0) = \psi_{TA}(\mathcal U_0)$, which will map back into $TA$. On the lower boundary we have $\mathcal U \equiv 1$, so when approaching that boundary, $\psi(1; \mathcal U_0) \to -\infty$ as $\mathcal U_0\to 1$, and the right-hand side of \eqref{eq:d_psi_d_zeta} is bounded, so $\psi(\zeta; \mathcal U_0) \to -\infty$ as $\mathcal U_0\to 1$. Therefore, the lower boundary will map back into the point $T$. Finally, the left boundary maps into the line $\varphi = t_{inj}+a_\zeta(0)x$, since $\zeta = 0$ on it, and as $\mathcal U$ approaches $\mathcal U_J(0)$, the values of $\psi(\zeta; \zeta_0)$ must go to infinity, since $\psi_\Phi(\zeta) \to \infty$ as $\zeta\to 0$ by the definition \eqref{eq:psi_Phi_definition}.
In summary, we demonstrated that the boundaries of $\overline{\Upsilon}_\Phi$ map back into the boundaries of $\angle_{t_{inj}}$, and since $\psi$ is continuous and monotone with respect to the parameters, the interior will also map one-to-one. Thus, the characteristics in the $(\varphi, x)$ plane cover the whole area $\angle_{t_{inj}}$.

The proof that $\mathcal U$, when mapped back into the $(\varphi, x)$ plane, gives us a piecewise continuously differentiable function can be repeated with no changes.
\end{proof}

\subsubsection{General case with finite number of sign changes}

The scheme of the previous case can be iterated any finite number of times to construct the solution with more changes of sign. We omit the step-by-step proof of such procedure, and only formulate the final result.

\begin{theorem}
\label{thm3}
Let the left-hand side of the condition ($\mathcal F$5) change sign only at a finite number of points. Then the $\mathcal U$-characteristics in $\angle_{t_{inj}}$ (continued from TA and constructed via the Jouguet condition \eqref{eq:Joguet_cond} from each point of the chemical shock front $(\Phi(x),x)$ such that ($\mathcal F$5) holds for $\zeta(\Phi(x),x)$) fill the whole area $\angle_{t_{inj}}$ without intersections and give us a pice-wise continuously differentiable solution $\mathcal U(\varphi, x)$.
\end{theorem}

The general case with infinite sign changes isn't fully clear right now. The working conjecture is that Theorem \ref{thm3} should hold in that case as well, but we haven't formally worked through the case when the points of sign change have a concentration point yet, so we leave it an open problem for now. 

\subsection{Solution for $\mathcal U$ everywhere else}
\subsubsection{Solution for $\mathcal U$ above the cone}
Since the whole cone is covered with $\mathcal U$-characteristics, we have a value for $\mathcal U$ at every point of the upper boundary $\varphi = t_{inj} + a_\zeta(0)x$. We can construct straight line $\mathcal U$-characteristics going above the cone, and they will cover the whole area. Moreover, the values of $\mathcal U$ on this boundary are mapped from $\{(\zeta, \mathcal U): \zeta=0, \mathcal U\in[1, \mathcal U_J(0)]\}$ with monotone values of $\psi$, therefore $\mathcal U(t_{inj} + a_\zeta(0)x, x)$ is monotone with respect to $x$, so the straight line characteristics, having the inclines $\mathcal F_{\mathcal U}(\mathcal U, 0)$, fan out and do not intersect.

\subsubsection{Solution for $\mathcal U$ below the chemical shock front}
The solution in the cone gives us values of $\mathcal U^+$ at all points immediately above the chemical shock front $\Phi$. The Rankine--Hugoniot condition \eqref{eq:RH-2} and the entropy condition (W4) allow us to calculate the value $\mathcal U^-$ below the shock. At the part of $OA$ the solution must coincide with the Riemann problem solution, which means that $\mathcal U^+_{TA}$ satisfies $\mathcal F_{\mathcal U}(\mathcal U^+_{TA}, 1) = a(1)$ and corresponds to the case $u_2^-=u_1^-$ in the original coordinates (see Sect.~\ref{sec:sec2-nullclines} for the definition of $u_{1,2}^\pm$), while $\mathcal U^-_{TA}$ corresponds to $u^+_1$ (due to known Riemann problem solution). This means that $\mathcal U^-$ on $\Phi$ must always correspond to $u_1^+$ in original coordinates, since a jump from $u_2^+$ to $u_1^+$ is inadmissible due to Oleinik's E-condition \eqref{Oleinik_admissibility}. Meanwhile, $\mathcal F_{\mathcal U}(\mathcal U^+, \zeta^+) \leqslant a(\zeta^+)/\zeta^+$ (due to the generalized variation of Lemma~\ref{lemma-U-Phi-above-U-J} for the case of Theorem~\ref{thm3}), so $\mathcal U^+$ always corresponds to $u_1^-$. The $c$-shock from $u_1^-$ to $u_1^+$ is admissible due to Proposition~\ref{prop:c-shock-admissibility}, so the solution we constructed satisfies (W4).

We can then construct the straight line characteristics using these values to cover the whole area below the shock. Note, though, that these characteristics are not proved to fan out (even though in most example cases they do), and if they collide, an additional $\mathcal U$-shocks could form. But, since this is an area with $\zeta\equiv 0$, the problem is, practically speaking, one-dimensional there, so these additional shocks can be constructed without breaking the entropy condition due to classical one-dimentional Kru\v{z}kov's theorem.

\subsubsection{Uniqueness of the constructed solution}
In summary, we constructed a piece-wise smooth solution with finite number of jumps, satisfying between them the classical differential equations \eqref{eq:U-lagr-eqn}, \eqref{eq:c-lagr-eqn} and thus \eqref{eq:main_system_chem_flood} in the original coordinates, so it satisfies (W1) and (W3). For the solution $\zeta$ we have explicit formulas, so we can trivially check that (W2) holds. Each constructed shock satisfies an entropy condition necessary for (W4) to hold. Therefore, we can use Theorem~\ref{uniqueness-theorem} and obtain uniqueness for our solution in the class of W-solutions.

\section{Examples}
\label{sec:Examples}

\subsection{Full Jouguet case}
\begin{figure}[H]
    \begin{center}
    \includegraphics[width=0.66\linewidth]{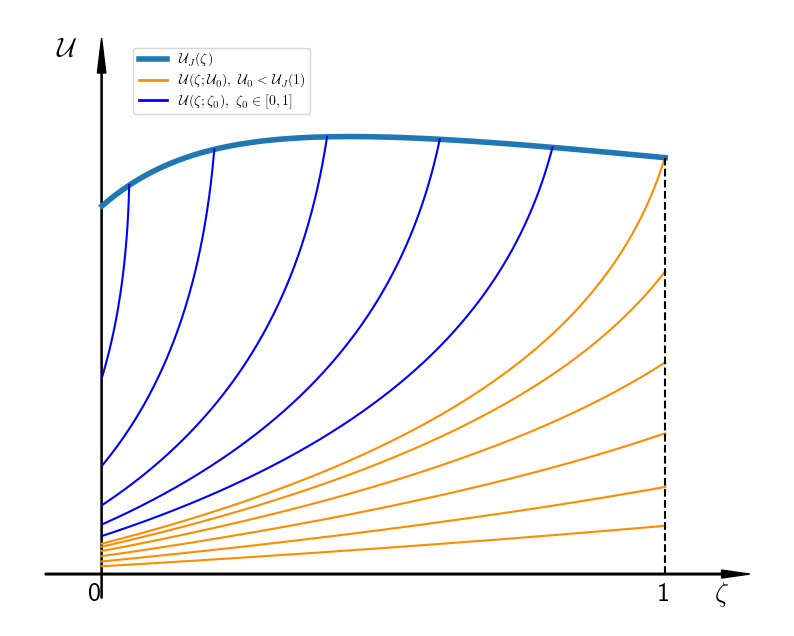}
    \end{center}
    \caption{Trajectories of the $\mathcal U$-characteristics in the $(\zeta, \mathcal U)$ plane. Full Jouguet case.}
    \label{fig:full_Jouguet_U_zeta}
    \begin{center}
    \includegraphics[width=0.66\linewidth]{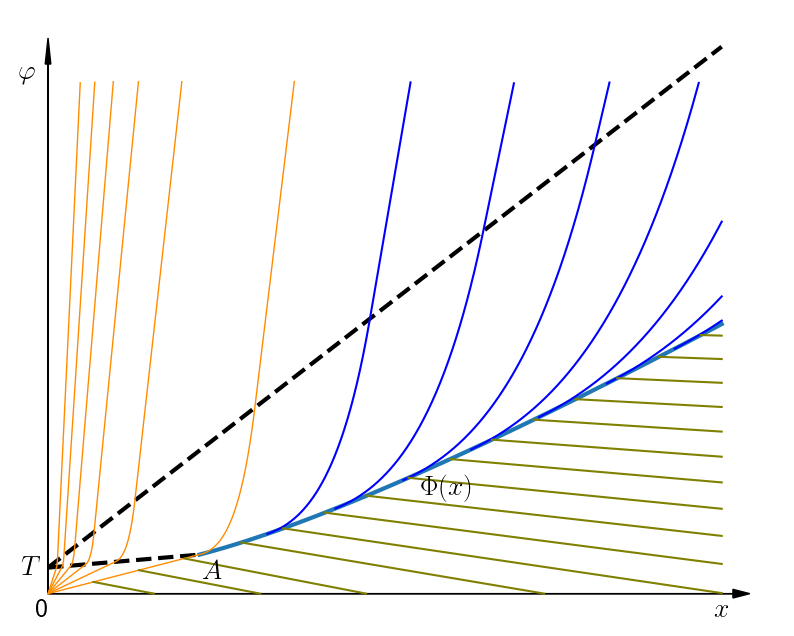}
    \end{center}
    \caption{Characteristics of the $\mathcal U$-solution in the $(\varphi, x)$ plane. Full Jouguet case. Colors match Figure~\ref{fig:full_Jouguet_U_zeta}.}
    \label{fig:full_Joguet_phi_x}
\end{figure}

\subsection{Partial Jouguet with one change of sign}
\begin{figure}[H]
    \begin{center}
    \includegraphics[width=0.66\linewidth]{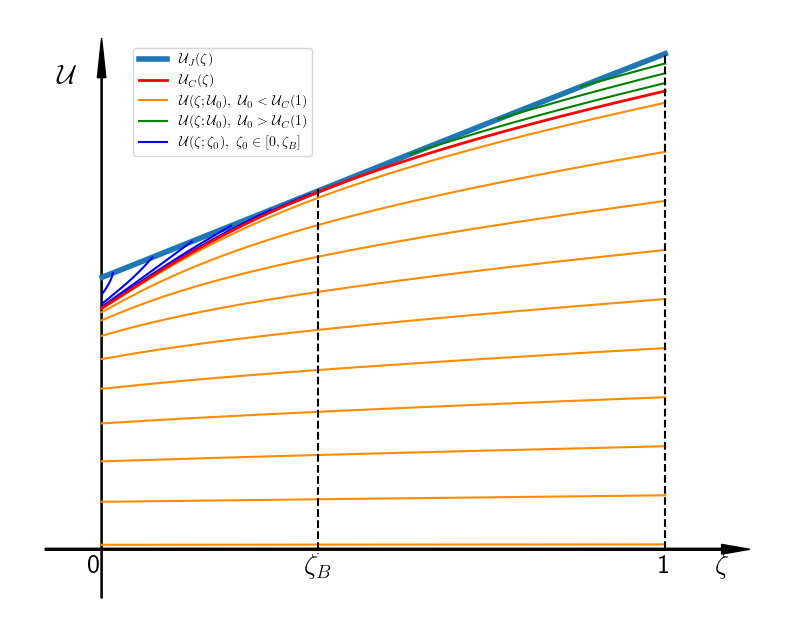}
    \end{center}
    \caption{Trajectories of the $\mathcal U$-characteristics in the $(\zeta, \mathcal U)$ plane. One change in the sign of the left-hand side of $(\mathcal F5)$.}
    \label{fig:one_sign_change_U_zeta}
    \begin{center}
    \includegraphics[width=0.66\linewidth]{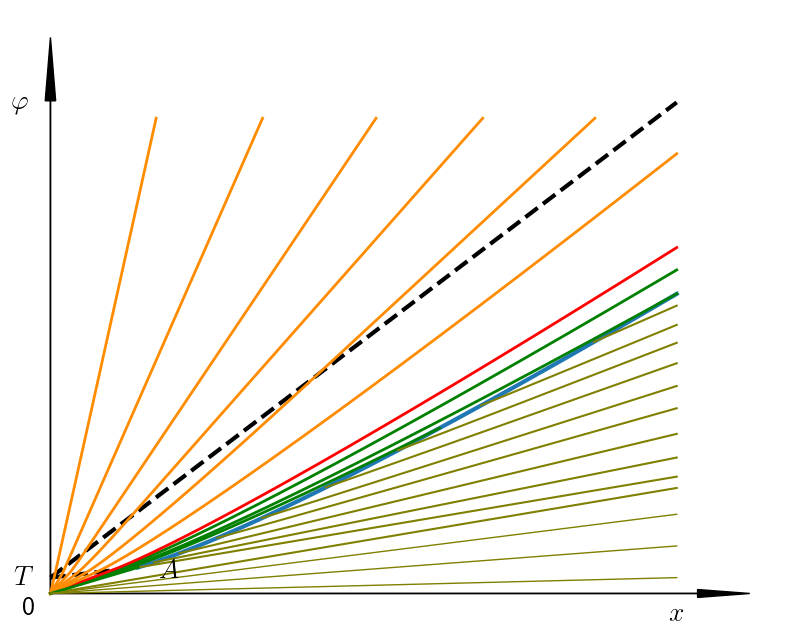}
    \end{center}
    \caption{Characteristics of the $\mathcal U$-solution in the $(\varphi, x)$ plane. One change in the sign of the left-hand side of $(\mathcal F5)$. The plot is framed to show the close neighborhood of the point $A$ and characteristics intersecting $\Phi$ near it. At a certain point $x_B$ the red characteristic touches the shock $\Phi$, and after that, Jouguet characteristics fill the rest of the area. Colors match Figure~\ref{fig:one_sign_change_U_zeta}.}
    \label{fig:one_sign_change_phi_x}
\end{figure}

\section*{Acknowledgements}

The authors thank Pavel Bedrikovetsky for lectures on systems of hyperbolic conservation laws.
Research is supported by the Russian Science Foundation (RSF) grant 19-71-30002.
\bigskip
\bigskip

\begin{enbibliography}{99}
\addcontentsline{toc}{section}{References}

\bibitem{Pires2020}
de O Apolin\'{a}rio, F., de Paula, A.~S. and Pires, A.~P., 2020. Injection of water slug containing two polymers in porous media: Analytical solution for two-phase flow accounting for adsorption effects. Journal of Petroleum Science and Engineering, 188, p.~106927.

\bibitem{Pires2021}
Apolin\'{a}rio, F.O. and Pires, A.P., 2021. Oil displacement by multicomponent slug injection: An analytical solution for Langmuir adsorption isotherm. Journal of Petroleum Science and Engineering, 197, p.~107939.

\bibitem{Bahetal}
Bakharev, F., Enin, A., Petrova, Y. and Rastegaev, N., 2023. Impact of dissipation ratio on vanishing viscosity solutions of the Riemann problem for chemical flooding model. Journal of Hyperbolic Differential Equations, 20(02), pp.~407-432.

\bibitem{Tapering}
Bakharev, F., Enin, A., Kalinin, K., Petrova, Y., Rastegaev, N. and Tikhomirov, S., 2023. Optimal polymer slugs injection profiles. Journal of Computational and Applied Mathematics, 425, p.~115042.

\bibitem{BL} Buckley, S.~E. and Leverett, M., 1942. Mechanism of fluid displacement in sands. Transactions of the AIME, 146(01), pp.~107-116.

\bibitem{Castaneda}
Casta\~{n}eda, P., 2016. Dogma: S-shaped. The Mathematical Intelligencer, 38, pp.~10-13.

\bibitem{Courant}
Courant, R., 1944. Supersonic Flow and Shock Waves: A Manual on the Mathematical Theory of Non-linear Wave Motion (No. 62). Courant Institute of Mathematical Sciences, New York University.

\bibitem{Dafermos}
Dafermos, C.~M., 2000. Hyperbolic Conservation Laws in Continuum
Physics. Springer Verlag, Berlin.

\bibitem{Gelfand}
Gelfand, I.~M., 1959. Some problems in the theory of quasilinear equations. Uspekhi Matematicheskikh Nauk, 14(2), pp.~87-158 (in Russian). English translation in Transactions of the American Mathematical Society, 29(2), 1963, pp.~295-381.

\bibitem{JnW}
Johansen, T. and Winther, R., 1988. The solution of the Riemann problem for a hyperbolic system of conservation laws modeling polymer flooding. SIAM journal on mathematical analysis, 19(3), pp.~541-566.

\bibitem{Kruzhkov}
Kru\v{z}kov, S.N., 1970. First order quasilinear equations in several independent variables. Mathematics of the USSR-Sbornik, 10(2), pp.~217-243.

\bibitem{Oleinik}
Oleinik, O.~A., 1957. Discontinuous solutions of non-linear differential equations. Uspekhi Matematicheskikh Nauk, 12(3)(75), pp.~3-73 (in Russian). English translation in American Mathematical Society Translations, 26(2), 1963, pp.~95-172.

\bibitem{PiBeSh06}
Pires, A.P., Bedrikovetsky, P.G. and Shapiro, A.A., 2006. A splitting technique for analytical modelling of two-phase multicomponent flow in porous media. Journal of Petroleum Science and Engineering, 51(1-2), pp.~54-67.

\bibitem{RastS-Shaped}
Rastegaev, N., 2023. On the sufficient conditions for the S-shaped Buckley-Leverett function. arXiv preprint arXiv:2303.16803.

\bibitem{MR2024}
Rastegaev, N. and Matveenko, S., 2024. Kru\v{z}kov-type uniqueness theorem for the chemical flood conservation law system with local vanishing viscosity admissibility. Journal of Hyperbolic Differential Equations, 21(04), pp.~1003-1043.

\bibitem{RheeAmundson}
Rhee, H.~K. and Amundson, N.~R., 1970. On the theory of multicomponent chromatography. Philosophical Transactions of the Royal Society of London. Series A, Mathematical and Physical Sciences, 267(1182), pp.~419-455.

\bibitem{Serre1}
Serre, D. Systems of Conservation Laws 1: Hyperbolicity, entropies, shock waves. Cambridge University Press, 1999.

\bibitem{Shen}
Shen, W., 2017. On the uniqueness of vanishing viscosity solutions for Riemann problems for polymer flooding. Nonlinear Differential Equations and Applications NoDEA, 24, pp.~1-25.

\bibitem{Wa87}
Wagner, D.H., 1987. Equivalence of the Euler and Lagrangian equations of gas dynamics for weak solutions. Journal of differential equations, 68(1), pp.~118-136. 

\end{enbibliography}

\end{document}

%% file: pic_Q_lagr_Q_orig_2.tex
\def\lowborder{-3}
\def\leftborder{-1}
\def\upperborder{6}
\def\rightborder{7}
\def\startpoint{0}
\def\phinote{0.8 * \lowborder}
\def\eps{1}
\begin{tikzpicture}[>=stealth', yscale=0.5, xscale=0.5]
\draw[thin,->] ({\leftborder}, 0) -- (\rightborder, 0) node[below] {$x\phantom{x^0}$};
\begin{scope}
\filldraw[fill=black!20, ultra thin]
(\startpoint, 0)  -- (\startpoint+6, 0)
  decorate [decoration={random steps,segment length=1pt,amplitude=1pt}] {-- (\startpoint + 6, \upperborder - \eps)}
 decorate [decoration={random steps,segment length=3pt,amplitude=1pt}] {--(0, \upperborder - \eps)}
 -- (0, 0) -- cycle;
\draw[very thick, blue] (\startpoint, 0)  -- (\startpoint+6, 0);
\node[] at (\rightborder * 0.45, 4) {$Q_{orig}$};
\node[below left] at (0, 0) {$0$};
\draw[thin,->] (0,\lowborder) -- (0,\upperborder) node[below right ] {$t$};
\end{scope}
\end{tikzpicture}
\hspace{1cm}
\begin{tikzpicture}[>=stealth', yscale=0.5, xscale=0.5]
\begin{scope}
\filldraw[fill=black!20, ultra thin]
(0, 0) -- (\rightborder - \eps, \phinote)
  decorate [decoration={random steps,segment length=1pt,amplitude=0.5pt}] {-- (\rightborder - \eps, \upperborder - \eps)}
 decorate [decoration={random steps,segment length=3pt,amplitude=1pt}] {--(0, \upperborder - \eps)}
 -- (0, 0) -- cycle;
 \draw[thin,->] ({\leftborder}, 0) -- (\rightborder, 0) node[below] {$x$};
\draw[very thick, blue] (0, 0)    --  (\rightborder - \eps, \phinote);
\node[] at (\startpoint + 2.5, -2) {$\varphi_0(x)$} ;
\node[] at (\rightborder * 0.45, 4) {$Q_{lagr}$};
\node[below left] at (0, 0) {$0$};
\draw[dashed]  (\startpoint, \phinote) -- (\startpoint, 0);
\draw[dashed] (\startpoint, \phinote) -- (0, \phinote);
\end{scope}
\draw[thin,->] (0,\lowborder) -- (0,\upperborder) node[below right ] {$\varphi$};
\end{tikzpicture}

%% file: pic_Q_lagr_Q_orig.tex
\def\lowborder{-2}
\def\leftborder{-1}
\def\upperborder{6}
\def\rightborder{7}
\def\startpoint{0}
\def\phinote{0.0 * \lowborder}
\def\eps{1}
\begin{tikzpicture}[>=stealth', yscale=0.5, xscale=0.5]
\draw[thin,->] ({\leftborder}, 0) -- (\rightborder, 0) node[below] {$x\phantom{x^0}$};
\begin{scope}
\filldraw[fill=black!20, ultra thin]
(\startpoint, 0)  -- (\startpoint+2, 1) .. controls (\startpoint+3, 1.5) .. (\startpoint + 4, 3) .. controls (\startpoint+4.8, 4) .. (\startpoint + 6, 4.5)
  decorate [decoration={random steps,segment length=1pt,amplitude=1pt}] {-- (\startpoint + 6, \upperborder - \eps)}
 decorate [decoration={random steps,segment length=3pt,amplitude=1pt}] {--(0, \upperborder - \eps)}
 -- (0, 0) -- cycle;
\draw[very thick, red] (\startpoint, 0)  -- (\startpoint+2, 1)  .. controls (\startpoint+3, 1.5) .. (\startpoint + 4, 3) .. controls (\startpoint+4.8, 4) .. (\startpoint + 6, 4.5);
\node[below right] at (\startpoint + 3.5, 3) {$t_0(x)$} ;
\node[] at (\rightborder * 0.45, 4) {$Q_{orig}$};
\node[below left] at (0, 0) {$0$};
\draw[thin,->] (0,\lowborder) -- (0,\upperborder) node[below right ] {$t$};
\end{scope}
\end{tikzpicture}
\hspace{1cm}
\begin{tikzpicture}[>=stealth', yscale=0.5, xscale=0.5]
\begin{scope}
\filldraw[fill=black!20, ultra thin]
(0, 0)  parabola[bend at end] (\startpoint * 0.2, \phinote * 0.3)
parabola (\startpoint * 0.8, \phinote *0.9) parabola[bend at end] (\startpoint, \phinote) -- (\rightborder - \eps, \phinote)
  decorate [decoration={random steps,segment length=1pt,amplitude=0.5pt}] {-- (\rightborder - \eps, \upperborder - \eps)}
 decorate [decoration={random steps,segment length=3pt,amplitude=1pt}] {--(0, \upperborder - \eps)}
 -- (0, 0) -- cycle;
 \draw[thin,->] ({\leftborder}, 0) -- (\rightborder, 0) node[below] {$x$};
\draw[very thick, blue] (0, 0)  parabola[bend at end] (\startpoint * 0.2, \phinote * 0.3)
parabola (\startpoint * 0.8, \phinote *0.9) parabola[bend at end] (\startpoint, \phinote);
\node[] at (\rightborder * 0.45, 4) {$Q_{lagr}$};
\node[below left] at (0, 0) {$0$};
\draw[very thick, red] (\startpoint, \phinote) -- (\rightborder - \eps, \phinote);
\draw[dashed]  (\startpoint, \phinote) -- (\startpoint, 0);
\draw[dashed] (\startpoint, \phinote) -- (0, \phinote);
\end{scope}
\draw[thin,->] (0,\lowborder) -- (0,\upperborder) node[below right ] {$\varphi$};
\end{tikzpicture}